\documentclass[sn-mathphys-num]{sn-jnl}

\usepackage{graphicx}%
\usepackage{multirow}%
\usepackage{amsmath,amssymb,amsfonts}%
\usepackage{amsthm}%
\usepackage{mathrsfs}%
\usepackage[title]{appendix}%
\usepackage{xcolor}%
\usepackage{textcomp}%
\usepackage{manyfoot}%
\usepackage{booktabs}%
\usepackage{algorithm}%
\usepackage{algorithmicx}%
\usepackage{algpseudocode}%
\usepackage{listings}%
\usepackage{hyperref, upgreek, mathtools, tikz, float}

\theoremstyle{thmstyleone}%
\newtheorem{theorem}{Theorem}

\newtheorem{proposition}{Proposition}%
\newtheorem{lemma}{Lemma}
\newtheorem{corollary}{Corollary}
\newtheorem*{question*}{Question}

\theoremstyle{thmstyletwo}%
\newtheorem{example}{Example}%
\newtheorem{remark}{Remark}%

\theoremstyle{thmstylethree}%
\newtheorem{definition}{Definition}%

\newcommand{\abs}[1]{\lvert#1\rvert}
\newcommand*\induce[1]{#1_\mathcal{K}}
\newcommand*\hyp[1]{\mathcal{#1}}

\begin{document}

\title[Article Title]{Mean dimension explosion of induced homeomorphisms}


\author*{\fnm{Gabriel} \sur{Lacerda}}\email{lacerda@im.ufrj.br}

\author*{\fnm{Sergio} \sur{Romaña}\footnote{The first author is a PhD candidate at UFRJ, funded by a CNPq grant, and gratefully acknowledges financial support for this research from the Fulbright U.S. Student Program, which is sponsored by the U.S. Department of State and Fulbright Brasil. The second author was supported by FAPERJ with the grant Bolsa Jovem Cientista do Nosso Estado No. E-26/201.432/2022, Brazil, NNSFC 12071202, and NNSFC 12161141002 from China.}}\email{sergio@mail.sysu.edu.cn}

\affil{\orgdiv{Mathematics Department}, \orgname{Federal University of Rio de Janeiro}, \orgaddress{\street{Av. Athos da Silveira Ramos 149}, \city{Rio de Janeiro}, \postcode{21945-970}, \state{RJ}, \country{Brazil}}}

\affil{\orgdiv{School of Mathematics (Zhuhai)}, \orgname{Sun Yat-sen University}, \postcode{519802}, \country{China}}


\abstract{Let $X$ be a compact metric space and $T: X \to X$ a continuous function. The induced hyperspace map $\induce{T}$ acts on the hyperspace $\hyp{K}(X)$ of closed and nonempty subsets of $X$, and on the continuum hyperspace $\hyp{C}(X) \subset \hyp{K}(X)$ of connected sets. This work studies the mean dimension explosion phenomenon: when the base system $T$ has zero topological entropy, but the mean dimension of the induced map $\induce{T}$ is infinite. In particular, this phenomenon occurs for Morse-Smale diffeomorphisms. Furthermore, for a circle homeomorphism $H$, the mean dimension explosion does not occur if and only if $H$ is conjugate to a rotation. For the metric mean dimension, a different result is obtained: we establish sufficient conditions for the induced hyperspace map to have zero or infinite metric mean dimension.}




\maketitle

\section{Introduction}

The study of dynamical systems is primarily concerned with the complexity of the orbits of a map $T: X \to X$, where $X$ is a compact metric space and $T$ is a continuous function. There exist numerous tools to evaluate this complexity, and the mean dimension is one of them. Proposed by M. Gromov in \cite{gromov_1999}, the \emph{mean dimension}, denoted as $\text{mdim}(X, T)$, is a topological invariant of dynamical systems, and it is useful to classify maps acting on infinite-dimensional topological spaces, as the mean dimension of $(X, T)$ is zero when the dimension of the phase space $X$ is finite.

E. Lindenstrauss and B. Weiss later developed the mean dimension in \cite{lindenstrauss_weiss_2000}, where they defined the metric mean dimension. It is a simpler tool to use compared to the original mean dimension, as evidenced by several recent results in the literature:\cite{hayes_2017, lindenstrauss_tsukamoto_2018, carvalho_rodrigues_varandas_2020, acevedo_romana_arias_2024}.

This work applies the theory of mean dimension and metric mean dimension to the induced hyperspace map $\induce{T}$ acting on the hyperspace of compact sets, $\hyp{K}(X)$, and on the hyperspace of subcontinua, $\hyp{C}(X)$. Since these hyperspaces are often infinite-dimensional, studying the mean dimension of the induced hyperspace map offers an alternative to classifying the complexity of well-known finite-dimensional dynamical systems. Moreover, the mean dimension appears as an alternative for evaluating the complexity of these toy models of infinite-dimensional dynamical systems.

In particular, we prove that a broad class of zero topological entropy dynamical systems acting on a finite-dimensional phase space have infinite mean dimension in the hyperspace, a phenomenon we term \textit{explosion}.

The investigation of the relationship between the base system $T$ and the induced map $\induce{T}$ began with W. Bauer and K. Sigmund in \cite{bauer_sigmund_1975}, where, in particular, they proved that \textit{if $\induce{T}: \hyp{K}(X) \to \hyp{K}(X)$ is topologically transitive, then so is $T: X \to X$}. Recently, M. Lampart and P. Raith established in \cite{lampart_raith_2010} a sufficient condition on $(X, T)$ for which the topological entropy $h(\induce{T})$ is infinite. Therefore, it is a valid question whether the metric mean dimension of an induced map is finite or infinite. Surprisingly, we present a sufficient condition on $T$ ensuring that the mean dimension of the induced map, $\text{mdim}(\hyp{K}(X), \induce{T})$, is infinite. For the following, a \emph{continuum} is a compact and connected metric space.\\

\begin{theorem}\label{thm:nonwandering-infty-mean-dim}
    Let $X$ be a locally connected continuum and $T: X \to X$ a homeomorphism. If the nonwandering set $\Omega(T)$ is a strict subset of $X$, then $\normalfont{\text{mdim}}(\hyp{K}(X), \induce{T}) = \infty$.\\
\end{theorem}

The identity map Id on $X$ satisfies $\Omega(\text{Id}) = X$ and $\text{mdim}(\hyp{K}(X), \induce{\text{Id}}) = 0$. Furthermore, in the special case of $X=S^1$, we exhibit a family of dynamics $H$ such that $\Omega(H)=S^1$ and $\text{mdim}(\hyp{K}(S^1), \induce{H})=0$, demonstrating that the condition in Theorem \ref{thm:nonwandering-infty-mean-dim} is sharp (see Corollary \ref{coro:rotation-number-mean-dimension}).

Theorem \ref{thm:nonwandering-infty-mean-dim} applies to a certain class of dynamical systems with zero topological entropy, which includes the Morse-Smale diffeomorphisms. In essence, the phenomenon of mean dimension explosion arises for the induced hyperspace map of these systems.

Another application of the above theorem is that every minimal dynamical system with finite mean dimension can be embedded in a subsystem of $\induce{T}$, when the nonwandering set of $T$ is strictly contained in the phase space. For context, Y. Gutman and M. Tsukamoto showed in \cite{gutman_tsukamoto_2020} that any minimal map with $\text{mdim}(X, T) < \frac{N}{2}$ can be embedded into the shift on $([0, 1]^N)^\mathbb{Z}$. Moreover, Theorem \ref{thm:nonwandering-infty-mean-dim} provides examples of maps that cannot be embedded into the shift on the Hilbert cube.

One could ask whether, for induced dynamics that attain infinite mean dimension, there exist subdynamics that achieve a mean dimension equal to any prescribed non-negative real number. The answer is positive, as indicated by the following result.\\

\begin{corollary}\label{coro:alpha_equal_mean_dimension_induced_explosion}
Let $X$ be a locally connected continuum and $T: X \to X$ a homeomorphism such that $\Omega(T) \subsetneq X$. Then, for every non-negative real number $\alpha$, there exists a $\induce{T}$-invariant set $Y_\alpha \subset \hyp{K}(X)$ such that $\emph{mdim}(Y_\alpha, \induce{T}|_{Y_\alpha}) = \alpha$.\\
\end{corollary}

In other words, Corollary \ref{coro:alpha_equal_mean_dimension_induced_explosion} states that $(X, T)$ has quasi-factors of any prescribed mean dimension. This interpretation uses the concept of quasi-factors, which was introduced by E. Glasner and B. Weiss in \cite{glasner_weiss_1995}.

In \cite[Proposition 6]{bauer_sigmund_1975}, the authors proved that \textit{if the topological entropy $h(T)$ is positive, then $h(\induce{T})$ is infinite}. This fact leads to the following question:\\  

\begin{question*}
    Let $X$ be a locally connected continuum and $T: X \to X$ a homeomorphism. If $h(T) > 0$, then we also have $\normalfont{\text{mdim}}(\hyp{K}(X),\induce{T}) = \infty$?\\
\end{question*}

The condition of being a locally connected continuum is necessary because D. Burguet and R. Shi gave in \cite{burguet_shi_2022} an example of a dynamical system $(X, T)$, where $X$ is zero-dimensional, such that the topological entropy $h(T)$ is positive, but $\text{mdim}(\hyp{K}(X), \induce{T})$ is zero. Furthermore, we show in Theorem \ref{thm:existence-compact-space-pos-entropy-zero-metric-mean-dim} the existence of a dynamical system such that the base map has positive topological entropy and the induced hyperspace map has zero metric mean dimension.

The continuum hyperspace $\hyp{C}(X)$ is an infinite-dimensional topological space when $X$ is a continuum with no free arcs. Given this, it is valid to ask if there is an analog of Theorem \ref{thm:nonwandering-infty-mean-dim} for such topological spaces $X$. In particular, A. Arbieto and J. Bohorquez proved in \cite[Theorem B]{arbieto_bohorquez_2023} that \textit{the topological entropy of the induced continuum map $\induce{F}: \hyp{C}(N^m) \to \hyp{C}(N^m)$ is either zero or infinite} where $F: N^m \to N^m$ is a Morse-Smale diffeomorphism acting on a connected and compact boundaryless $m$-dimensional manifold $N^m$. Our second result deals with the mean dimension of the induced continuum map of Morse-Smale dynamics. In fact, we establish the following result:\\

\begin{theorem}\label{thm:continuum-hyp-morse-smale}
    Let $F: N^m \to N^m$ be a Morse-Smale diffeomorphism. Hence, the following dichotomy holds:
    \begin{itemize}
        \item if $m = 1$, then $\normalfont{\text{mdim}}(\hyp{C}(N^m), \induce{F}) = 0$;
        \item if $m > 1$, then $\normalfont{\text{mdim}}(\hyp{C}(N^m), \induce{F}) = \infty$.
    \end{itemize}
\end{theorem}

The hypothesis of Theorem \ref{thm:continuum-hyp-morse-smale} is weaker than the hypothesis of Theorem \ref{thm:nonwandering-infty-mean-dim} because the nonwandering set of a Morse-Smale diffeomorphism is finite. As a consequence, Theorem \ref{thm:continuum-hyp-morse-smale} states that the mean dimension explosion also occurs for the induced continuum map, when the dimension of the phase space is greater than one.

Another kind of dichotomy holds for the mean dimension as well. It is a stronger version of the following fact: \textit{if $H$ is a homeomorphism on $S^1$, then the topological entropy of $\induce{H}: \hyp{K}(S^1) \to \hyp{K}(S^1)$ is either zero or infinite}. This fact is proved in \cite[Theorem 5]{lampart_raith_2010}. Based on this result, we have the following statement:\\

\begin{theorem}\label{thm:dichotomy-circle-homeo}
    Given $H: S^1 \to S^1$ a homeomorphism, then $\normalfont{\text{mdim}}(\hyp{K}(S^1), \induce{H})$ is either $0$ or $\infty$.\\
\end{theorem}

The ideas used in the proof of Theorem \ref{thm:dichotomy-circle-homeo} are similar to those presented in \cite{lampart_raith_2010}, where the presence of periodic orbits plays a crucial role. Furthermore, since orientation-preserving homeomorphisms on the circle are well understood, we can utilize Theorem \ref{thm:dichotomy-circle-homeo} to classify them based on the mean dimension of their induced hyperspace map. In other words, the phenomenon of mean dimension explosion is fully understood for circle homeomorphisms, since all circle homeomorphisms have zero topological entropy. \\

\begin{corollary}\label{coro:rotation-number-mean-dimension}
    Let $H: S^1 \to S^1$ be an orientation-preserving homeomorphism, then $\normalfont{\text{mdim}}(\hyp{K}(S^1), \induce{H}) = 0$ if, and only if, $H$ is conjugated to a rotation.\\
\end{corollary}

Recently, B. Zhu, X. Huang, and X. Wang proved in \cite{zhu_huang_wang_2024} that if the mean dimension of the base system $T$ is positive, then the mean dimension of the induced hyperspace system $\induce{T}$ is infinite. We would like to point out that this does not explain the phenomenon of mean dimension explosion, as this phenomenon specifically occurs when the base system is defined on a finite-dimensional phase space and has zero topological entropy.

It is important to state that the mean dimension explosion phenomenon does not occur for the induced system of Borel probability measures $\hyp{M}(X)$ equipped with the push-forward map $T_*: \mu \mapsto \mu\circ T^{-1}$, where $T: X \to X$ is a continuous map, because D. Burguet and R. Shi proved in \cite{burguet_shi_2022} that $h(T) > 0$ if, and only if, $\text{mdim}(\hyp{M}(X), T_*) = \infty$. 

In the latter part of this work, we study the \emph{metric mean dimension} of induced maps. If $d$ is a metric on $X$, the metric mean dimension is denoted by $\text{mdim}_M(X, d, T)$. The standard metric for both hyperspaces, which coincides with the natural topology on $\hyp{K}(X)$, is the Hausdorff metric $d_H$. In our first result of this part, we investigate under which conditions on $T$ the upper metric mean dimension of the induced hyperspace map $\induce{T}$ is zero. This condition is given in terms of the \emph{polynomial entropy} of $T$, $h_{\text{pol}}(T)$, a refinement of the topological entropy for maps with low complexity. It was defined by C. Labrousse and J.P. Marco in \cite{labrousse_marco_2014}. Our result relies on the fact that any locally connected continuum admits a convex metric, which is a type of geodesic metric.\\

\begin{theorem}\label{thm:poly-entropy-less1-metric-mean-dim-zero}
    Consider $X$ a locally connected continuum endowed with a convex metric $d$ and $T: X \to X$ a map. If $h_{\emph{pol}}(T) < 1$, then $\emph{mdim}_M(\hyp{K}(X), d_H, \induce{T}) = 0$.\\
\end{theorem}

In fact, the above result follows from a stronger statement: if the polynomial entropy of $T$ is strictly less than $1$, then the topological entropy of the induced hyperspace map $\induce{T}$ is zero. This statement is proved in a lemma in this work.

A simpler condition for obtaining zero metric mean dimension for the induced hyperspace map $\induce{T}$ is that the base map $T$ is an isometry. This implies that $\induce{T}$ is an isometry too, and hence the complexity of its orbits is low; therefore, it does not have a positive upper metric mean dimension.

A valid question is whether the converse of the above theorem holds. The converse does not hold for a general compact metric space, as the following result shows.\\

\begin{theorem}\label{thm:existence-compact-space-pos-entropy-zero-metric-mean-dim}
    There exist a compact metric space $(\Lambda, \rho)$ and a map $T: \Lambda \to \Lambda$ such that $h(T) > 0$ and $\emph{mdim}_M(\hyp{K}(\Lambda), \rho_H, \induce{T}) = 0$.\\
\end{theorem}

Since the polynomial entropy is a finite number only if the topological entropy is zero, the above result shows that the converse of Theorem \ref{thm:poly-entropy-less1-metric-mean-dim-zero} does not hold for a compact metric space. In fact, the behavior of the metric mean dimension for induced hyperspace maps is more complex when the base space is a general compact metric space. Indeed, X. Huang and X. Wang proved in \cite{huang_wang_2022} that for the shift map $\sigma$ on the symbolic space of $N$ symbols, $\Sigma_N$, endowed with the usual metric $d(\xi,\eta)=2^{-\min\{n\ge1:\,\xi_n\ne \eta_n\}}$, one has $\text{mdim}(\hyp{K}(\Sigma_N),d_H,\induce{\sigma})=\infty$.\\

\begin{question*}
    Let $X$ be a locally connected continuum endowed with a convex metric $d$, and let $T: X \to X$ be a map. If $h(T) > 0$, does it follow that $\emph{mdim}_M(\hyp{K}(X), d_H, \induce{T}) = \infty$? 
    
    Alternatively, if $X$ is a compact metrizable space, is there always some metric $\rho$ on $X$ such that $h(T) > 0$ implies $\emph{mdim}_M(\hyp{K}(X), \rho_H, \induce{T}) = \infty$?\\
\end{question*}

The hypothesis on $X$ being a locally connected continuum is crucial because the topological dimension of $\hyp{K}(X)$ is infinite only if $X$ is this type of topological space. Hence, it is possible for the upper metric mean dimension of $\induce{T}$ to be infinite, since it is bounded above by the box-counting dimension of the phase space. Moreover, it is important to note that the metric mean dimension depends on the choice of metric for the phase space.

As a final result, we present a slight generalization of Theorem \ref{thm:nonwandering-infty-mean-dim}. More precisely, if we do not require the map $T$ to be injective, then metric mean dimension of $\induce{T}$ is infinite.\\

\begin{theorem}\label{thm:surjective-map-nonwandering-infinite-metric-mean-dim}
    Let $X$ be a locally connected continuum endowed with a convex metric $d$. If $T: X \to X$ is a surjective map and $\Omega(T) \subsetneq X$, then $\emph{mdim}_M(\hyp{K}(X), d_H, \induce{T}) = \infty$.\\
\end{theorem}

We point out that the condition of the nonwandering set being strictly contained in the phase space does not preclude a dynamical system from having positive topological entropy, since there are numerous examples of maps with both positive entropy and wandering sets.

\subsection{Reading guide}

This work is organized as follows: Section \ref{sec:basic-def} provides the fundamental definitions needed to understand the theory discussed in this paper. Section \ref{sec:mean-dimension-induced-systems} is dedicated to studying induced hyperspace maps of homeomorphisms and their mean dimension; specifically, \S\ref{subsec:proof-thm-infty-mean-dim-closed-hyp} presents the proof of Theorem \ref{thm:nonwandering-infty-mean-dim} and Corollary \ref{coro:alpha_equal_mean_dimension_induced_explosion}, and \S\ref{subsec:proof-thm-morse-smale-continuum-hyp} provides the proof of Theorem \ref{thm:continuum-hyp-morse-smale}. The mean dimension explosion is completely understood for circle homeomorphisms, and \S\ref{subsec:explosion-circle} includes the proofs of Theorem \ref{thm:dichotomy-circle-homeo} and Corollary \ref{coro:rotation-number-mean-dimension}. Finally, Section \ref{sec:metric-mean-dim-induced} focuses on the metric mean dimension for the induced hyperspace maps, where we present the proofs of Theorems \ref{thm:poly-entropy-less1-metric-mean-dim-zero}, \ref{thm:existence-compact-space-pos-entropy-zero-metric-mean-dim}, and \ref{thm:surjective-map-nonwandering-infinite-metric-mean-dim}.


\section{Basic Definitions}\label{sec:basic-def}

From now on, let $(X, d)$ be a compact metric space, and let $T: X \to X$ be a continuous function, or simply, a map. A dynamical system is the pair $(X, T)$. In specific cases, $X$ is a nondegenerate locally connected continuum. In other words, $(X, d)$ is a Peano continuum, which is arcwise connected; that is, any two points of $X$ can be joined by an arc in $X$. Moreover, $d$ can be chosen to be a convex metric. Precisely, a convex metric $d$ on $X$ is a metric that induces the topology on $X$ and for which midpoints always exist. That is, for any $x, y \in X$, there exists $z \in X$ such that
\begin{equation*}
    d(x, z)  = \frac{1}{2}d(x, y) = d(z, y).
\end{equation*}
For further information on locally connected continua and convex metrics, see \cite[Section 10]{illanes_nadler_1999}. To set some notation, for a given $\varepsilon > 0$, define $B_\varepsilon(x) = \{y \in X : d(x, y) < \varepsilon\}$. 

To avoid excessive use of parentheses, we write $Tx$ for the image of $x \in X$ under the map $T$. We say that a closed set $Y \subset X$ is \emph{$T$-invariant} if, for every $y \in Y$, $Ty \in Y$, that is, $T(Y) \subset Y$. In this case, $(Y, T|_Y)$ is also a dynamical system. Moreover, $Y$ is \emph{totally invariant} if $T(Y) = Y$. An important set in the study of dynamical systems is the \textit{nonwandering set} of $T$, denoted by $\Omega(T)$. A point $x \in \Omega(T)$ if, for every neighborhood $U$ of $x$ in $X$, there exists an integer $n \geq 1$ such that $T^n(U) \cap U \neq \varnothing$. Equivalently, there exists $m \geq 1$ such that $T^{-m}(U) \cap U \neq \varnothing$. The nonwandering set is nonempty, compact, and invariant.

In general, if $E$ is a finite set, we denote its cardinality by $\#E$.


\subsection{Hyperspaces and induced maps}\label{section:hyperspaces-induced-maps}

A \textit{hyperspace} is a designated collection of subsets of $X$. In this work, we study the following hyperspaces:

\begin{itemize}
    \item $\hyp{K}(X) = \{A \subset X; A \text{ is closed and nonempty} \}$;
    \item $\hyp{C}(X) = \{A \in \hyp{K}(X); A \text{ is connected}\}$, when $X$ is a continuum.
\end{itemize}

Observe that $\hyp{C}(X) \subset \hyp{K}(X)$. In this text, we designate these as the \textit{compacta hyperspace} and the \textit{continuum hyperspace}, respectively. We always assume that $X$ is not a singleton. In both hyperspaces, we consider the \emph{Hausdorff metric} $d_H$ defined as 
\begin{equation*}
    d_H(A, B) = \inf\{\varepsilon \geq 0; A \subset B_\varepsilon \text{ and } B \subset A_\varepsilon\},
\end{equation*}
where $A_\varepsilon = \bigcup_{x \in A} \{y \in X; d(x, y) \leq \varepsilon\}$ is the \emph{$\varepsilon$-fattening of A}. In this context, given $\varepsilon > 0$, define $\textbf{B}^H(A, \varepsilon) = \{B \in \hyp{K}(X); d_H(A, B) < \varepsilon\}$. Moreover, whenever $X$ is a compact topological space, the topology generated by $d_H$ coincides with the \emph{Vietoris topology}, whose basis consists of the following collection of sets indexed by $k \in \mathbb{N}$:
\begin{equation*}
    \langle U_1, \dots, U_k \rangle = \{A \in \hyp{K}(X); A \subset \bigcup_{i=1}^k U_i \text{ and } A \cap U_i \neq \varnothing \text{ for each } 1 \leq i \leq k\},
\end{equation*}
where $U_1, \dots, U_k$ are nonempty open subsets of $X$.


It is known that $(\hyp{K}(X),d_H)$ is a compact metric space \cite[Theorem 3.5]{illanes_nadler_1999}. Moreover, if $X$ is a nondegenerate locally connected continuum, then $\hyp{K}(X)$ is homeomorphic to the Hilbert cube $[0, 1]^{\mathbb{Z}}$. If, in addition, $X$ has no free arcs\footnote{A continuum $X$ has no free arcs if, for every $\gamma \subset X$ homeomorphic to the unit interval $[0, 1]$, the interior of $\gamma$ is not open in $X$.}, then $\hyp{C}(X)$ is also homeomorphic to the Hilbert cube (cf. \cite{curtis_schori_1974}). In both cases, the topological dimension of $\hyp{K}(X)$ and $\hyp{C}(X)$ is infinite.

Let $F: X \to X$ be a map, the \textit{induced hyperspace map} $\induce{F}: \hyp{K}(X) \to \hyp{K}(X)$ is given by $\induce{F}(A) := F(A)$. Note that $F(A)$ is closed in $X$, because $F$ is continuous, and therefore $\induce{F}$ is well defined. It is also well known that $\induce{F}$ is continuous, and a homeomorphism if and only if $F$ is a homeomorphism. The \textit{induced continuum map} is intuitively defined as the map $\induce{F}$ restricted to $\hyp{C}(X)$, that is, $\induce{F}: \hyp{C}(X) \to \hyp{C}(X)$. An alternate notation is $\induce{F}|_{\hyp{C}(X)}: \hyp{C}(X) \to \hyp{C}(X)$.\\




\subsection{Mean dimension}

Consider $\mathcal{U} = \{U_i\}_{i \in I}$ a finite open cover of $X$. The \textit{order} $\text{ord}(\mathcal{U})$ is the maximum integer $k \geq 0$ such that there are pairwise distinct $i_0, ..., i_k \in I$ satisfying $U_{i_0} \cap ... \cap U_{i_k} \neq \varnothing$. A \textit{refinement} of $\mathcal{U}$ is an open cover $\mathcal{V} = \{V_j\}_{j \in J}$ of $X$ such that for every $V_j \in \mathcal{V}$ there is $U_i \in \mathcal{U}$ with $V_j \subset U_i$. The \textit{topological dimension} of $X$ is given by the supremum of the \textit{degree} $\mathcal{D}(\mathcal{U})$, which is the minimum order $\text{ord}(\mathcal{V})$ over all refinements $\mathcal{V}$ of $\mathcal{U}$ (see Definition 1.6.7 in \cite{engelking_1978} for details). The dimension of a topological space is denoted by $\dim(X)$. In general, the topological dimension coincides with the usual notion of dimension for more familiar spaces, such as manifolds. Moreover, $\dim([0,1]^{\mathbb{Z}}) = \infty$.

For two open covers $\mathcal{U}$ and $\mathcal{V}$ of $X$, the \textit{joint} is given by $\mathcal{U} \vee \mathcal{V} = \{U_i \cap V_j; i \in I, j \in J\}$. Since $T:X \to X$ is continuous, observe that $T^{-1}\mathcal{U} = \{T^{-1}U_i; i \in I\}$ is also an open cover of $X$. Therefore, the \textit{(topological) mean dimension} of $(X, T)$ is given by
\begin{equation*}
    \text{mdim}(X, T) = \sup\limits_{\mathcal{U}} \lim\limits_{n \to \infty} \frac{\mathcal{D}(\mathcal{U} \vee T^{-1}\mathcal{U}\vee \cdots \vee T^{-n+1}\mathcal{U})}{n},
\end{equation*}
where $\mathcal{U}$ runs over all finite open covers of $X$. Observe that the mean dimension is the dynamical analog of the topological dimension.\\

\begin{proposition}[\cite{lindenstrauss_weiss_2000, coornaert_2015}]
    Some useful basic properties of the mean dimension are:
\begin{enumerate}
    \item[1.] The mean dimension is a topological invariant and takes values in $[0, \infty]$;
    \item[2.] If the topological entropy of the dynamical system is finite, then its mean dimension is zero;
    \item[3.] If $Y$ is a closed $T$-invariant subset of $X$, then $\emph{mdim}(Y, T) \leq \emph{mdim}(X, T)$;
    \item[4.] If $X$ is finite-dimensional, then $\emph{mdim}(X, T) = 0$;
    \item[5.] For any dynamical system $(X, T)$,  $\emph{mdim}(X, T^n) = n\cdot\emph{mdim}(X, T)$;
    \item[6.] If the phase space is the Hilbert cube $[0, 1]^\mathbb{Z}$, and $\sigma: [0, 1]^\mathbb{Z} \to [0, 1]^\mathbb{Z}$ is the shift transformation, then $\emph{mdim}([0, 1]^\mathbb{Z}, \sigma) = 1$. Generally, $\emph{mdim}(([0, 1]^d)^\mathbb{Z}, \sigma) = d$.
\end{enumerate}
\end{proposition}


\subsection{Topological entropy and Metric mean dimension}

Given $n \in \mathbb{N}$, define the \emph{dynamical distance} $d_n$ as 
\[
d_n(x, y) = \max \{d\big(T^ix, T^iy\big), 0 \leq i \leq n - 1\},
\]
where $T^0$ is the identity. It is well known that $d_n$ is still a distance function on $X$ and generates the same topology as $d$. For $x \in X$ and $\varepsilon > 0$, we call $B_{(n, \varepsilon)}(x) = \{y \in X; d_n(x, y) < \varepsilon\}$ the \textit{$(n, \varepsilon)$-dynamical ball around x}.

Call $A \subset X$ a \textit{$(n, \varepsilon)$-separated set} if for any distinct points $x, y \in A$, $d_n(x, y) \geq \varepsilon$. Denote by $\text{Sep}(T, n, \varepsilon)$ the maximal cardinality of a $(n, \varepsilon)$-separated set, which is finite by the compactness of $X$. We say that $E \subset X$ is a \textit{$(n, \varepsilon)$-spanning set} for $X$ if for any $x \in X$, there exists $y \in E$ such that $d_n(x, y) < \varepsilon$. In other words, $E$ is a $(n, \varepsilon)$-spanning set if $X$ can be covered by the interior of the union of $(n, \varepsilon)$-dynamical balls centered at points of $E$. That is,
\begin{equation*}
    X \subset  \bigcup\limits_{y \in E} B_{(n, \varepsilon)}(y).
\end{equation*}
Let $\text{Span}(T, n, \varepsilon)$ be the minimum cardinality of any $(n, \varepsilon)$-spanning set. Define 
\[
h(T, \varepsilon) = \limsup\limits_{n \to \infty} \frac{\log \text{Sep}(T, n, \varepsilon)}{n} \text{ and } \widetilde{h}(T, \varepsilon) = \limsup\limits_{n \to \infty} \frac{\log \text{Span}(T, n, \varepsilon)}{n}.
\]

It is well known that $h(T, \varepsilon) \geq \widetilde{h}(T, \varepsilon)$ (cf. \cite[Chapter 3]{katok_hasselblatt_1995}). Note that if $\varepsilon_1 < \varepsilon_2$, then $\widetilde{h}(T, \varepsilon_1) \geq \widetilde{h}(T, \varepsilon_2)$. The \textit{topological entropy} $h(T)$ is the limit of both $h(T, \varepsilon)$ and $\widetilde{h}(T, \varepsilon)$ as $\varepsilon \to 0$, which coincide. That is,
\begin{equation*}
    h(T) = \lim_{\varepsilon \to 0} h(T, \varepsilon) = \lim_{\varepsilon \to 0} \widetilde{h}(T, \varepsilon).
\end{equation*}
When the topological entropy is infinite, we are interested in the growth of $\widetilde{h}(T, \varepsilon)$ with respect to $\varepsilon$. That is, the growth of the rate of exponential balls which cover $X$ in mean. This motivates the definition of upper and lower metric mean dimension, given in \cite{lindenstrauss_weiss_2000}.

The \textit{lower metric mean dimension} and the \textit{upper metric mean dimension} of $(X, T)$ are defined, respectively, by
\begin{equation}\label{eq:def-metric-mean-dim}
\underline{\text{mdim}_M}(X, d, T) = \liminf\limits_{\varepsilon \to 0} \frac{h(T, \varepsilon)}{-\log\varepsilon}
\;\;\;\text{and}\;\;\;\overline{\text{mdim}_M}(X, d, T) = \limsup\limits_{\varepsilon \to 0} \frac{h(T, \varepsilon)}{-\log\varepsilon}.
\end{equation}
In these definitions, $h(T, \varepsilon)$ can be replaced with $\widetilde{h}(T, \varepsilon)$. When both limits coincide, we denote this common value by $\text{mdim}_M(X, d, T)$. Unlike topological entropy, the metric mean dimension may change if the metric $d$ is changed.

The following important inequalities relate the mean dimension, the metric mean dimension, and the topological entropy whenever $d$ is a metric compatible with the topology on $X$ \cite[Theorem 4.2]{lindenstrauss_weiss_2000}:
\begin{equation}\label{eq:mean-metric-mean-inequality}
    \text{mdim}(X, T) \leq \underline{\text{mdim}_M}(X, d, T) \leq \overline{\text{mdim}_M}(X, d, T) \leq h(T).
\end{equation}
Note that if $h(T)$ is finite, then the metric mean dimension is zero. Therefore, the metric mean dimension is useful for distinguishing between dynamical systems with infinite topological entropy.

In contrast, the concept of \emph{polynomial entropy} is useful for distinguishing between dynamical systems with topological entropy equal to zero. It was introduced by C. Labrousse and J.P. Marco in \cite{labrousse_marco_2014} and is given by:
\begin{equation}
    h_{\text{pol}}(T) = \lim\limits_{\varepsilon \to 0}\limsup\limits_{n \to \infty} \frac{\log \text{Sep}(T, n, \varepsilon)}{\log n} = \lim\limits_{\varepsilon \to 0}\limsup\limits_{n \to \infty} \frac{\log \text{Span}(T, n, \varepsilon)}{\log n}.
\end{equation}


\subsection{Morse-Smale diffeomorphisms.}\label{subsec:morse-smale-diffeos}

For $N^m$ an $m$-dimensional compact and connected smooth manifold without boundary, set $\text{Diff}^r(N^m)$ as the set of $C^r$ diffeomorphisms endowed with the $C^r$ topology, for $r \geq 1$. A periodic point $x \in N^m$ of period $k \geq 1$ for $F \in \text{Diff}^r(N^m)$ is \textit{hyperbolic} if the derivative $(DF^k)_x$ has its spectrum disjoint from the unit circle in $\mathbb{C}$. It is well known that in this case, we have the existence of stable and unstable manifolds of $x$, denoted by $W^s(x)$ and $W^u(x)$. \\

\begin{definition}
    $F \in \emph{Diff}^r(N^m)$ is \textit{Morse-Smale} if it satisfies the following conditions:
    \begin{enumerate}
        \item The set of nonwandering points, $\Omega(F)$, contains only a finite number of hyperbolic periodic points;
        \item The stable and unstable manifolds of the periodic points are all transversal to each other.
    \end{enumerate}
\end{definition}

The Morse-Smale diffeomorphisms have the following properties:
\begin{itemize}
    \item Every Morse-Smale diffeomorphism has an \textit{attractor} periodic point and a \textit{repeller} periodic point, that is, the stable (unstable) manifold of the attractor (repeller) periodic point is an immersed submanifold of dimension $m$;
    \item If $p$ is an attractor periodic point, then there is $q$ a repeller periodic point such that the stable manifold of $p$ intersects transversely the unstable manifold of $q$ and $W^s(p) \cap W^u(q) \neq \varnothing$;
    \item The topological entropy of a Morse-Smale diffeomorphism is always zero.
\end{itemize}

Further information on Morse-Smale diffeomorphisms can be found at \cite{arbieto_bohorquez_2023}.

\section{Mean dimension of induced homeomorphisms}\label{sec:mean-dimension-induced-systems}

This section is devoted to the study of the topological mean dimension of induced hyperspace maps for homeomorphisms defined in a locally connected continuum. A topological space is said to be \emph{arcwise connected} provided that any two of its points can be joined by an arc contained in this space.\\

\begin{proposition}[\cite{illanes_nadler_1999}, Theorem 10.2]\label{prop:loc-con-conti-arcwise-connected}
    Every locally connected continuum is arcwise connected.
\end{proposition}

\subsection{Proof of Theorem \ref{thm:nonwandering-infty-mean-dim}}\label{subsec:proof-thm-infty-mean-dim-closed-hyp}

Recall that from our hypothesis we suppose that $X$ is a locally connected continuum and $T: X \to X$ is a homeomorphism such that $\Omega(T) \subsetneq X$. The following proof relies on the fact that there is a wandering arc $\gamma_x$ contained in $X\setminus\Omega(T)$, which is guaranteed by Proposition \ref{prop:loc-con-conti-arcwise-connected}, and we show how to use it to increase the mean dimension of the hyperspace by identifying the hyperspace of the arc with the Hilbert cube.

\begin{proof}[Proof of Theorem \emph{\ref{thm:nonwandering-infty-mean-dim}}]
Remember that the nonwandering set is compact, hence $X\setminus\Omega(T)$ is an open set. Given $x \in X\setminus\Omega(T)$, there is $U_x$ a neighborhood of $x$ such that $T^n(U_x) \cap U_x = \varnothing$, where $n$ is a non-zero integer, because $T$ is a homeomorphism. As a consequence, $T^{n_1}(U_x) \cap T^{n_2}(U_x) = \varnothing$, for $n_1 \neq n_2$ in $\mathbb{Z}$. Let $\gamma_x$ be an arc contained in $U_x$, which exists because $X$ is a locally connected continuum \cite[Theorem 10.2]{illanes_nadler_1999}. Thus, consider $\varphi: [0, 1] \to \gamma_x$ a homeomorphism and $k \in \mathbb{N}$. Choose $k$ disjoint closed intervals $J_i$ in $[0, 1]$. For each $J_i$, $1 \leq i \leq k$, consider a homeomorphism $\psi_i: [0, 1] \to J_i$. If $x = (x_1, ..., x_k)$ is a point in $[0, 1]^k$, then $\varphi(\{\psi_i(x_i), 1 \leq i \leq k\}) \in \hyp{K}(X)$. Set $\psi(x) = \{\psi_i(x_i), 1 \leq i \leq k\}$, hence $\varphi\circ\psi(x)$ is an element of $\hyp{K}(X)$. Denote $\gamma_i = \varphi(J_i)$ a subset of $\gamma_x$.

    Define a map $\Phi: ([0, 1]^k)^\mathbb{Z} \to \hyp{K}(X)$, for $\xi = (..., \xi_{-1}, \xi_0, \xi_1, ...) \in ([0, 1]^k)^\mathbb{Z}$, as
    \begin{equation*}
        \Phi(\xi) = \Phi(..., \xi_{-1}, \xi_0, \xi_1, ...) = \left(\bigcup\limits_{n \in \mathbb{Z}} T^n\circ\varphi\circ\psi(\xi_n)\right) \cup \Omega(T).
    \end{equation*}

    Observe that the limit points of $\bigcup\limits_{n \in \mathbb{Z}} T^n\circ\varphi\circ\psi(\xi_n)$ are contained in the nonwandering set $\Omega(T)$, because $T^n\circ\varphi\circ\psi(\xi_n) \subset T^n(\gamma_x)$, for each $n \in \mathbb{Z}$, and the limit points of the wandering arc $\gamma_x$ are contained in the nonwandering set. Thus, $\Phi(\xi)$ is an element of the hyperspace.

    Given that $\induce{T}(\Omega(T)) = \Omega(T)$, we also have

    \begin{center}
    \begin{tabular}{ c c l }
     $\induce{T}\circ\Phi(\xi)$ & $=$ & $\left(\bigcup\limits_{n \in \mathbb{Z}} T^{n + 1}\circ\varphi\circ\psi(\xi_n)\right) \cup \Omega(T)$ \\ 
      & $=$ & $\left(\bigcup\limits_{n \in \mathbb{Z}} T^n\circ\varphi\circ\psi(\xi_{n - 1})\right) \cup \Omega(T)$ \\  
      & $=$ & $\Phi\circ\sigma^{-1}(\xi),$    
    \end{tabular}
    \end{center}   
    where $\sigma: ([0, 1]^k)^\mathbb{Z} \to ([0, 1]^k)^\mathbb{Z}$ is the shift transformation. Lindenstrauss \& Weiss proved in \cite{lindenstrauss_weiss_2000} that $\text{mdim}(([0, 1]^k)^\mathbb{Z}, \sigma) = k$. Since the mean dimension is a topological invariant, the rest of this proof is dedicated to proving that $\Phi$ is injective and continuous. Note that the above equation also shows that $\Phi\big(([0, 1]^k)^\mathbb{Z}\big)$ is invariant by $\induce{T}$.

    It is not hard to prove that $\Phi$ is injective because if $\xi, \eta \in ([0, 1]^k)^\mathbb{Z}$ are such that $\xi \neq \eta$, then there is $\ell \in \mathbb{Z}$ such that $\xi_\ell \neq \eta_\ell$. That is, there is $i \in \{1, ..., k\}$ such that $\xi_{\ell, i} \neq \eta_{\ell, i}$. Then, $\psi_i(\xi_{\ell, i}) \neq \psi_i(\eta_{\ell, i})$. Therefore, by construction, $\Phi(\xi) \neq \Phi(\eta)$.

    To prove continuity, consider $\xi \in ([0, 1]^k)^\mathbb{Z}$ and $\varepsilon > 0$. There exists $N \in \mathbb{N}$ such that for all $n \in \mathbb{Z}$ where $\abs{n} \geq N$, the distance between $T^n\circ\varphi\circ\psi(\xi_n)$ and $\Omega(T)$ is less than $\varepsilon/4$, that is, there is $z \in \Omega(T)$ such that $d(z, T^n\circ\varphi\circ\psi(\xi_n)) < \varepsilon/4$. Such $N$ exists because the limit points of the sequence $(T^n\circ\varphi\circ\psi(\xi_n))_{n \in \mathbb{Z}}$ are all contained in $\Omega(T)$. Let $d_*$ be a metric in $([0, 1]^k)^\mathbb{Z}$ given by 
    \begin{equation*}
        d_*(\xi, \eta) = \sum\limits_{n \in \mathbb{Z}} 2^{-\abs{n}}\lVert\xi_n - \eta_n\rVert,
    \end{equation*}
    where $\lVert.\rVert$ is the Euclidean metric in $[0, 1]^k$. It is well known that such a metric is compatible with the topology of $([0, 1]^k)^\mathbb{Z}$ \cite{vanMill_1989}. Let $\delta > 0$ and $\eta \in ([0, 1]^k)^\mathbb{Z}$ be such that $d_*(\xi, \eta) < \delta$ and $\xi_i = \eta_i$ for $\abs{i} \leq N$. Note that $\bigcup\limits_{\abs{i} \leq N} T^i\circ\varphi\circ\psi(\xi_i) = \bigcup\limits_{\abs{i} \leq N} T^i\circ\varphi\circ\psi(\eta_i)$. Therefore 

    \begin{center}
    \begin{tabular}{ l l l }
    $d_H(\Phi(\xi), \Phi(\eta))$ & $=$ & $d_H(\Phi(\xi)\setminus\bigcup\limits_{\abs{i} \leq N} T^i\circ\varphi\circ\psi(\xi_i), \Phi(\eta)\setminus\bigcup\limits_{\abs{i} \leq N} T^i\circ\varphi\circ\psi(\eta_i))$\\
    & $\leq$ & $d_H(\Phi(\xi)\setminus\bigcup\limits_{\abs{i} \leq N} T^i\circ\varphi\circ\psi(\xi_i), \Omega(T)) + d_H(\Omega(T), \Phi(\eta)\setminus\bigcup\limits_{\abs{i} \leq N} T^i\circ\varphi\circ\psi(\eta_i))$\\
     & $\leq$ & $\varepsilon/4 + \varepsilon/4 < \varepsilon$,
    \end{tabular}
    \end{center}
    because $\Omega(T) \subset \Phi(\xi)$ and $\Omega(T) \subset \Phi(\eta)$.

    Finally, $\Phi$ is a homeomorphism of $([0, 1]^k)^\mathbb{Z}$ onto its image. Hence
    \begin{equation*}
        \text{mdim}(\hyp{K}(X), \induce{T}) \geq \text{mdim}(\Phi(([0, 1]^k)^\mathbb{Z}), \induce{T}) = k.
    \end{equation*}
    Since $k \in \mathbb{N}$ is arbitrary, then $\text{mdim}(\hyp{K}(X), \induce{T}) = \infty$.
\end{proof}

Whenever the nonwandering set is finite, the topological entropy of $T$ is zero \cite[Proposition 1 \& 2]{lampart_raith_2010}. Therefore, Theorem \ref{thm:nonwandering-infty-mean-dim} proves that for homeomorphisms with finite nonwandering set, the explosion phenomenon occurs. The most common examples of these systems are the Morse-Smale diffeomorphisms, defined in §\ref{subsec:morse-smale-diffeos}. The following consequence is trivial if one recalls Equation (\ref{eq:mean-metric-mean-inequality}).\\

\begin{remark}
    Let $X$ be a locally connected continuum and $T: X \to X$ a homeomorphism. If the nonwandering set $\Omega(T)$ is a strict subset of $X$, then $\normalfont{\text{mdim}}(\hyp{K}(X),\rho, \induce{T}) = \infty$, where $\rho$ is any metric compatible with the topology of $\hyp{K}(X)$. In particular, $\emph{mdim}(\hyp{K}(X), d_H, \induce{T}) = \infty$.\\
\end{remark}

To show Corollary \ref{coro:alpha_equal_mean_dimension_induced_explosion}, we need to recall a useful result.\\

\begin{proposition}[\cite{coornaert_2015}, Corollary 7.6.3]
    Let $d$ be a positive integer and $\alpha$ be a real number such that $0 \leq \alpha \leq d$. Then there exists a subshift $\mathfrak{Z} \subset ([0, 1]^d)^\mathbb{Z}$ such that $\emph{mdim}(\mathfrak{Z}, \sigma|_\mathfrak{Z}) = \alpha$.\\
\end{proposition}

Note that for every real number $\alpha \geq 0$, there is always some $k \in \mathbb{N}$ such that $k \geq \alpha$. By the above proposition, there exists a subshift $\mathfrak{Z} \subset ([0, 1]^k)^\mathbb{Z}$ such that $\text{mdim}(\mathfrak{Z}, \sigma|_\mathfrak{Z}) = \alpha$. Therefore, by the proof of Theorem \ref{thm:nonwandering-infty-mean-dim}, $\Phi(\mathfrak{Z})$ is a compact $\induce{T}$-invariant subset of $\hyp{K}(X)$ such that $\text{mdim}(\Phi(\mathfrak{Z}), \induce{T}|_{\Phi(\mathfrak{Z})}) = \alpha$.

\subsection{Proof of Theorem \ref{thm:continuum-hyp-morse-smale}}\label{subsec:proof-thm-morse-smale-continuum-hyp}

    The core of its proof relies on the same idea as in Theorem \ref{thm:nonwandering-infty-mean-dim}: We connect the wandering arcs through their elements to form a connected set. Each connected set will be an arc, and hence an element of the continuum hyperspace.
    
\begin{proof}[Proof of Theorem \emph{\ref{thm:continuum-hyp-morse-smale}}]
    Given $N^m$ a one-dimensional connected and compact boundaryless manifold, that is, a manifold homeomorphic to the circle $S^1$, then, the topological entropy $h(\induce{F})$ restricted to $\hyp{C}(N^m)$ is zero for any homeomorphism $F$ \cite[Theorem 1 \& 4]{lampart_raith_2010}. In this case, $\overline{\text{mdim}}(\hyp{C}(N^m), d_H, \induce{F}) = 0$. Therefore, $\text{mdim}(\hyp{C}(N^m), \induce{F}) = 0$.

    From now on, suppose that $N^m$ is a compact manifold of dimension $m$ greater than one. Since $F$ is a Morse-Smale diffeomorphism, then there is, up to an iterate of $F$, $p$ and $q$ hyperbolic fixed points such that $W^s(p) \cap W^u(q) \neq \varnothing$. Just as in the proof of Theorem \ref{thm:nonwandering-infty-mean-dim}, consider $\gamma \subset W^s(p) \cap W^u(q)$ an arc such that $F^{n_1}(\gamma) \cap F^{n_2}(\gamma) = \varnothing$, for $n_1 \neq n_2$ in $\mathbb{Z}$.

    By the Hartman-Grobman Theorem, consider $V$ a neighborhood of $p$ such that there is a homeomorphism $\psi: V \to \psi(V) \subset \mathbb{R}^m$ where $0 \in \psi(V)$ and $(DF)_p\circ\psi = \psi\circ F$. Note that we can shrink $V$ so that $F(V) \subset V$ and $\psi(V)$ is a convex set. Let $K \in \mathbb{N}$ be such that $F^K(\gamma) \subset V$; hence $\psi\circ F^K(\gamma)$ is a curve in $\mathbb{R}^m$. Observe that $\psi\circ F^{K + 1}(\gamma)$ is also a curve in $\psi(V)$.

    Given $z \in \psi\circ F^K(\gamma)$ and $w \in \psi\circ F^{K + 1}(\gamma)$, set $\beta(z, w, t)$ as the straight line connecting $z$ to $w$, with $\beta(z, w, 0) = z$ and $\beta(z, w, 1) = w$. Therefore, $\beta(z, w, t)$ is, in particular, a continuous function on the first two coordinates. Define the curve $\kappa(x, y, t) = F^{-K}\circ \psi^{-1}(\beta(z, w, t))$ that connects a point $x$ in $\gamma$ to $y$ in $F(\gamma)$, then $\Gamma:\gamma\times F(\gamma) \to \hyp{C}(N^m)$ set as $\Gamma(x, y) = \{\kappa(x, y, t), t \in [0, 1]\}$ is a continuous function on both coordinates.

    Let $k \in \mathbb{N}$ and consider $\gamma_j \subset \gamma$ connected and closed sets, for $1 \leq j \leq k$. Since each $\gamma_j$ is also a curve, there is $\tau_j: [0, 1] \to \gamma_j$ a homeomorphism. Given $\xi \in ([0, 1]^k)^{\mathbb{Z}}$, define $\Phi: ([0, 1]^k)^{\mathbb{Z}} \to \hyp{C}(N^m)$ as 
    \begin{equation*}
        \Phi(..., \xi_{-1}, \xi_0, \xi_1, ...) = \bigcup\limits_{j = 1}^k \overline{ \bigcup\limits_{i \in \mathbb{Z}}  F^i(\Gamma(\tau_j(\xi_{i, j}), F\circ\tau_j(\xi_{i + 1, j})))},
    \end{equation*}
    where $\xi_i = (\xi_{i, 1}, ..., \xi_{i, k})$ and $\xi_{i, j} \in [0, 1]$. The point $\Phi(\xi) \in \hyp{K}(N^m)$ is indeed a connected set of $N^m$ because, for all $i \in \mathbb{Z}$, 
    \begin{equation*}
        F^i(\Gamma(\tau_j(\xi_{i, j}), F\circ\tau_j(\xi_{i + 1, j}))) \cap F^{i + 1}(\Gamma(\tau_j(\xi_{i + 1, j}), F\circ\tau_j(\xi_{i + 2, j}))) = \{F^{i + 1}\circ\tau_j(\xi_{i + 1, j})\},
    \end{equation*}
    and also $\lim_{i \to \infty} F^{i + 1}\circ\tau_j(\xi_{i + 1, j}) \to p$, and $\lim_{i \to -\infty} F^{i + 1}\circ\tau_j(\xi_{i + 1, j}) \to q$, since $F^{i + 1}\circ\tau_j(\xi_{i + 1, j}) \in F^{i + 1}(\gamma)$. Therefore, $\Phi(\xi)$ is a finite union of connected sets with a common point in both $p$ and $q$, hence a connected set.

    To illustrate, for each $j \in \{1, ..., k\}$, the curve $\overline{ \bigcup\limits_{i \in \mathbb{Z}}  F^i(\kappa(\tau_j(\xi_{i, j}), F\circ\tau_j(\xi_{i + 1, j})))}$ connect $q$ to $p$ and pass through the sequence $(F^i\circ\tau_j(\xi_{i, j}))_{i \in \mathbb{N}}$ in $N^m$, as represented in Figure \ref{fig:proof-cont-hyp-infty-mean-dim}.

    \begin{figure}[ht]
        \centering
        \begin{tikzpicture}[scale=0.8]
    \draw[line width=1.5pt, color=red] (7.5,2.5) to[in=50, out=260] (7.5,0.5);
    \node[scale = 1.2] at (8,0.5) {$\gamma_1$};
    \draw[line width=1.5pt, color=red] (7.5,-1) to[in=100, out=260] (7.5,-3);
    \node[scale = 1.2] at (7.9,-3) {$\gamma_2$};

    \draw[line width=1.5pt, color=red] (9.5,1.8) to[in=60, out=240] (9.5,0.5);
    \node[scale = 0.8] at (9.7,0.2) {$F(\gamma_1)$};
    \draw[line width=1.5pt, color=red] (9.5,-1) to[in=60, out=280] (9.5,-2.5);
    \node[scale = 0.8] at (9.5,-2.8) {$F(\gamma_2)$};

    \draw[line width=1.5pt, color=red] (5.5,2) to[in=50, out=260] (5.5,1);
    \node[scale = 0.8] at (5.5,0.6) {$F^{-1}(\gamma_1)$};
    \draw[line width=1.5pt, color=red] (6,-1) to[in=100, out=260] (6,-2);
    \node[scale = 0.8] at (6,-2.3) {$F^{-1}(\gamma_2)$};

    \draw[line width=1pt, color=blue] (0,0) to[in=230, out=30] (4,2.5);
    \draw[line width=1pt, color=blue] (4,2.5) to[in=170, out=290] (7.6,1.5);
    \draw[line width=1pt, color=blue] (7.6,1.5) to[in=200, out=0] (13,0.5);
    \draw[line width=1pt, color=blue] (13,0.5) to[in=200, out=0] (15,0);

    \draw[line width=1pt, color=blue] (0,0) to[in=230, out=30] (5,-1);
    \draw[line width=1pt, color=blue] (5,-1) to[in=230, out=0] (7.45,-1.5);
    \draw[line width=1pt, color=blue] (7.45,-1.5) to[in=230, out=0] (11,-1.5);
    \draw[line width=1pt, color=blue] (11,-1.5) to[in=240, out=350] (15,0);
    
    \fill (7.6,1.5) circle[radius=3pt];
    \node[scale = 1] at (8.4,1.9) {$\tau_1(\xi_{0, 1})$};
    \fill (7.45,-1.5) circle[radius=3pt];
    \node[scale = 1] at (8.3,-1.15) {$\tau_2(\xi_{0, 2})$};

    \fill (0,0) circle[radius=2pt];
    \node[below, scale = 1.2] at (0,-0.2) {$q$};
    \fill (15,0) circle[radius=2pt];
    \node[below, scale = 1.2] at (15,-0.2) {$p$};
\end{tikzpicture}
        \caption{The connected set in blue is an example of $\Phi(\xi)$ for $k = 2$.}
        \label{fig:proof-cont-hyp-infty-mean-dim}
    \end{figure}
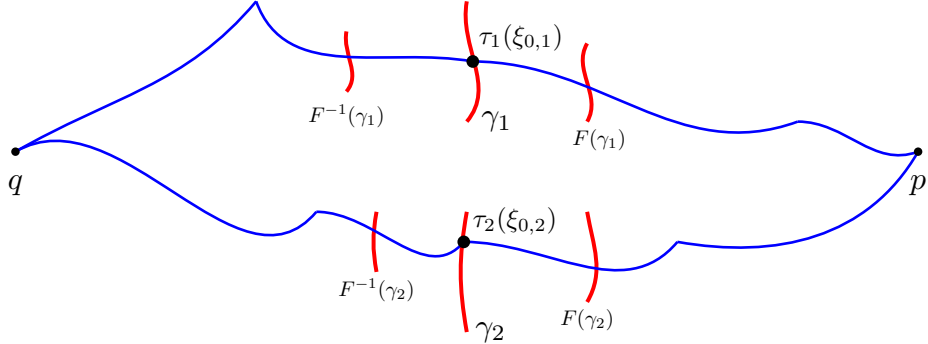
    
    Observe that
    \begin{equation*}
        \induce{F}\left(\overline{ \bigcup\limits_{i \in \mathbb{Z}}  F^i(\kappa(\tau_j(\xi_{i, j}), F\circ\tau_j(\xi_{i + 1, j})))}\right) = \overline{ \bigcup\limits_{i \in \mathbb{Z}}  F^{i + 1}(\kappa(\tau_j(\xi_{i, j}), F\circ\tau_j(\xi_{i + 1, j})))}.
    \end{equation*}
    Thus, $\induce{F}\circ\Phi(\xi) = \Phi\circ\sigma^{-1}(\xi)$.

    If $\Phi$ is a homeomorphism between $([0, 1]^k)^{\mathbb{Z}}$ and its image, then, since the mean dimension is a topological invariant,  $\text{mdim}(\hyp{C}(N^m), \induce{F}) \geq \text{mdim}(\Phi(([0, 1]^k)^{\mathbb{Z}}), \induce{F}) \geq k$, because $\text{mdim}(([0, 1]^k)^{\mathbb{Z}}, \sigma) = k$. Therefore, $\text{mdim}(\hyp{C}(N^m), \induce{F}) = \infty$, because $k \in \mathbb{N}$ is arbitrary. Thus, our task is reduced to prove that $\Phi$ is continuous and injective.

    To prove that $\Phi$ is injective, consider $\xi \neq \eta$ in $([0, 1]^k)^{\mathbb{Z}}$, hence, there is $i \in \mathbb{Z}$ such that $\xi_i \neq \eta_i$. Specifically, there is $j \in \{1, ..., k\}$ such that $\xi_{i, j} \neq \eta_{i, j}$. Therefore, $\Gamma(\tau_j(\xi_{i, j}), F\circ\tau_j(\xi_{i + 1, j})) \neq \Gamma(\tau_j(\eta_{i, j}), F\circ\tau_j(\eta_{i + 1, j}))$ as elements in $\hyp{C}(N^m)$. Since $F$ is, in particular, a homeomorphism, then $F^i(\Gamma(\tau_j(\xi_{i, j}), F\circ\tau_j(\xi_{i + 1, j}))) \neq F^i(\Gamma(\tau_j(\eta_{i, j}), F\circ\tau_j(\eta_{i + 1, j})))$. Therefore, the curves $\overline{ \bigcup\limits_{i \in \mathbb{Z}}  F^i(\Gamma(\tau_j(\xi_{i, j}), F\circ\tau_j(\xi_{i + 1, j})))}$ and $\overline{ \bigcup\limits_{i \in \mathbb{Z}}  F^i(\Gamma(\tau_j(\eta_{i, j}), F\circ\tau_j(\eta_{i + 1, j})))}$ are different elements of $\hyp{C}(N^m)$. Thus, $\Phi(\xi) \neq \Phi(\eta)$.

    To prove the continuity of $\Phi$, consider $(\xi^{(n)})_{n \in \mathbb{N}}$ converging to $\xi$ in $([0, 1]^k)^{\mathbb{Z}}$ with the product topology, that is, for each $i \in \mathbb{Z}$ and $j \in \{1, ..., k\}$, $\xi^{(n)}_{i,j} \to \xi_{i, j}$. The continuity of $\Gamma$ and $\tau_j$ guarantees that $\Gamma(\tau_j(\xi^{(n)}_{i,j}), F\circ\tau_j(\xi^{(n)}_{i + 1,j}))$ converges to $\Gamma(\tau_j(\xi_{i,j}), F\circ\tau_j(\xi_{i + 1,j}))$ as $n \to \infty$ and for each $j \in \{1, ..., k\}$. Then $F^i(\Gamma(\tau_j(\xi^{(n)}_{i,j}), F\circ\tau_j(\xi^{(n)}_{i + 1,j})))$ goes to $F^i(\Gamma(\tau_j(\xi_{i,j}), F\circ\tau_j(\xi_{i + 1,j})))$ for all $i \in \mathbb{Z}$. Therefore, $\Phi(\xi^{(n)}) \to \Phi(\xi)$.
\end{proof}

\subsection{Explosion on the circle}\label{subsec:explosion-circle}

The explosion of the mean dimension for the induced hyperspace map is better understood for homeomorphisms defined on one-dimensional topological manifolds because they all have zero topological entropy. Precisely, there is a dichotomy for such homeomorphisms: the mean dimension of its induced hyperspace map is zero or infinite. The following lemma states this dichotomy for interval homeomorphisms and is crucial to proving the analogous statement for circle homeomorphisms.\\

\begin{lemma}\label{lemma:t-square-not-identity}
    Let $T: [0, 1] \to [0, 1]$ be a homeomorphism such that $T^2$ is not the identity, then $\normalfont{\text{mdim}}(\hyp{K}([0, 1]), \induce{T}) = \infty$.
\end{lemma}
\begin{proof}
    The first step is to prove that $T^2$ is increasing. In fact, since $T$ is a homeomorphism, $T$ is a strictly increasing or decreasing function. If $T$ is increasing, then it is obvious. If $T$ is decreasing, then $Ta < Tb$, when $a > b$ in $[0, 1]$. Therefore, $T^2a > T^2b$. 

    Since $T^2$ is not the identity, then $T^2x < x$ or $T^2x > x$, for some $x \in [0, 1]$. Suppose, without loss of generality, that $T^2x < x$. Hence $T^{2n_2}x < T^{2n_1}x$ for $n_2 > n_1$ in $\mathbb{Z}$.

    Notice that $(T^{2n}x)_{n \geq 0}$ is a decreasing sequence, then $T^{2n}x$ converges to a fixed point of $T^2$. The same happens to $(T^{2k}x)_{k \leq 0}$, because $T^{2k}x$ is an increasing sequence. Recall that $T^2$ has at least two fixed points because $T^2$ is increasing.

    The second step is to choose a closed interval $J$ contained in the interval $(T^2x, x)$, therefore $T^{2n_2}(J) \cap T^{2n_1}(J) = \varnothing$ for $n_2 > n_1$ in $\mathbb{Z}$, and $\lim_{n \to \pm\infty}T^{2n}(J)$ is a fixed point of $T^2$. Since $J$ is also a wandering interval, just as in the proof of Theorem \ref{thm:nonwandering-infty-mean-dim}, we have $\text{mdim}(\hyp{K}([0, 1]), \induce{T}^2) = \text{mdim}(\hyp{K}([0, 1]), \induce{T}) = \infty$.
\end{proof}

\begin{remark}\label{rmrk:dychotomy-mean-dim-unit-interval}
    If $T^2: [0,1] \to [0, 1]$ is the identity, then $\normalfont{\text{mdim}}(\hyp{K}([0, 1]), \induce{T}^2) = \normalfont{\text{mdim}}(\hyp{K}([0, 1]), \induce{T}) = 0$. This proves the dichotomy for homeomorphisms defined in the unit interval $[0, 1]$.\\
\end{remark}

Before proceeding, we would like to point out the following remark about the hyperspace map of isometries.\\

\begin{remark}\label{prop:isometry-hyp}
    If $R: X \to X$ is an isometry, then the induced map $\induce{R}: \hyp{K}(X) \to \hyp{K}(X)$ is also an isometry. In fact, since $R$ is an isometry, for every $x_1, x_2 \in X$ we have $d(Rx_1, Rx_2) = d(x_1, x_2)$. Now let $A,B \subset X$ be closed sets such that $d_H(A,B) = \varepsilon$. Because $A \subset B_\varepsilon$, it follows that $R(A) \subset R(B)_\varepsilon$. If $\varepsilon_0 \in (0,\varepsilon)$, then $A \not\subset B_{\varepsilon_0}$, so there exist $a \in A$ and $b \in B$ with $d(a,b) > \varepsilon_0$. Thus $d(Ra,Rb) = d(a,b) > \varepsilon_0$, which shows that $R(A) \not\subset R(B)_{\varepsilon_0}$. Therefore no change occurs in the Hausdorff distance, and we conclude that  $d_H(\induce{R}(A),\induce{R}(B)) = d_H(A,B)$.\\
\end{remark}

To prove the explosion phenomenon on the circle $S^1$, we will refer to the consequences of the rotation number theory of circle homeomorphisms. A great exposition of this subject is given in \cite{defaria_guarino_2022}. Recall that two dynamical systems $T: X \to X$ and $F: Y \to Y$ are \textit{topologically conjugated} if there is a homeomorphism $\psi: X \to Y$ such that $\psi\circ F = T\circ \psi$.

\begin{proof}[Proof of Theorem \emph{\ref{thm:dichotomy-circle-homeo}}]
    Suppose that $H$ is a circle homeomorphism with a periodic point, then there exists $x \in S^1$ and $q \in \mathbb{N}$ such that $H^qx = x$. Every $z \in S^1$ can be written as $z = x\cdot e^{2\pi it}$ for $t \in [0, 1]$. The function $\psi: S^1 \to [0, 1]$ defined as $\psi(z) = t$ conjugates $H^q$ to a homeomorphism $T$ on $[0, 1]$ given by $T = \psi\circ H^q\circ \psi^{-1}$. If $T^2$ is the identity, then $H^{2q}$ is the identity. Thus, $\text{mdim}(\hyp{K}(S^1), \induce{H}) = 0$. Otherwise, by Lemma \ref{lemma:t-square-not-identity}, there is a $J \subset [0, 1]$ wandering interval for $T^2$, then $\psi^{-1}(J)$ is a wandering interval for $H^{2q}$. Therefore, by Theorem \ref{thm:nonwandering-infty-mean-dim}, $\text{mdim}(\hyp{K}(S^1), \induce{H}) = \infty$. 

    The last step is to consider that $H$ has no periodic points. If $H$ is conjugated to an irrational rotation $R_\theta$ by the homeomorphism $\phi$, then it is not hard to show that $\induce{H}$ is conjugated to $\induce{(R_\theta)}$ by the homeomorphism $\induce{\phi}$. Since $\induce{(R_\theta)}$ is an isometry, by Remark \ref{prop:isometry-hyp}, and the topological entropy is preserved by conjugation, the entropy of $\induce{H}$ is zero. Therefore, $\text{mdim}(\hyp{K}(S^1), \induce{H}) = 0$. Otherwise, the non-wandering set of $H$, $\Omega(H)$, is a Cantor set \cite[Proposition 2.5]{defaria_guarino_2022}. In this case, by the proof of Theorem \ref{thm:nonwandering-infty-mean-dim}, it is well known that there is a closed wandering interval $J \subset S^1\setminus \Omega(H)$, that is, its images by $H$ are pairwise disjoint, such that $\abs{H^n(J)} \to 0$ and $H^n(J) \to \Omega(H)$ when $\abs{n} \to \infty$. Then, $\text{mdim}(\hyp{K}(S^1), \induce{H}) = \infty$.
\end{proof}

Given $H: S^1 \to S^1$ an orientation-preserving homeomorphism, if $H$ is conjugated to a rotation, then $H$ is an isometry with respect to some metric compatible with the usual topology on $S^1$. Therefore, the metric mean dimension of its induced hyperspace map is zero. Conversely, if $\text{mdim}(\hyp{K}(S^1), \induce{H}) = 0$, we have two cases:
\begin{itemize}
    \item if $H$ has no periodic points, then $H$ is conjugated to an irrational rotation, by Theorem \ref{thm:dichotomy-circle-homeo};
    \item if $H$ has a periodic point of period $q \in \mathbb{N}$, then all periodic orbits of $H$ have period $q$ \cite[Proposition 2.4]{defaria_guarino_2022}. Suppose that $H^q$ is not the identity; then $T := 
    \psi\circ H^q\circ \psi^{-1}$, as in the proof of Theorem \ref{thm:dichotomy-circle-homeo}, is not the identity. By hypothesis, $T^2$ is the identity, hence $H^{2q}$ is the identity. This contradicts the period of its periodic orbits. Thus, $H^q$ must be the identity. In this case, it is well known that $H$ is topologically conjugated to a rational rotation \cite{constantin_kolev_1994}.
\end{itemize}

As a consequence, this discussion proves Corollary \ref{coro:rotation-number-mean-dimension}.\\

\begin{remark}
    The orientation-preserving condition on Corollary \emph{\ref{coro:rotation-number-mean-dimension}} is necessary because if $H$ is a reflection, then $\normalfont{\text{mdim}}(\hyp{K}(S^1), \induce{H}) = 0$ because $H^2$ is the identity. However, $H$ has two fixed points and is not conjugated to the identity. Therefore, $H$ is not conjugated to a rotation.
\end{remark}


\subsection{Differences between mean dimension and topological dimension}\label{subsec:difference-mean-top-dim}

As discussed above, the mean dimension is the dynamical analog of the topological dimension. The reader may wonder if results from classical dimension theory also hold for the mean dimension. In the context of hyperspaces, for instance, a classical result states that for any $T_1$-space\footnote{A \emph{$T_1$ space} is a topological space in which, for every pair of distinct points, each has a neighborhood not containing the other point.} $X$, we have $\dim(X) = 0$ if and only if $\dim(\hyp{K}(X)) = 0$ \cite[Proposition 8.6]{illanes_nadler_1999}.

Surprisingly, this result's analog does not hold for the mean dimension. Indeed, as a consequence of our previous results, there exists a dynamical system $(X, T)$ on a finite-dimensional phase space (hence $\text{mdim}(X, T) = 0$) such that $\text{mdim}(\hyp{K}(X), \induce{T}) = \infty$. However, the converse is true, and this is an easy consequence of a property of mean dimension.\\

\begin{proposition}
    If $(X, d)$ is a compact metric space, $T: X \to X$ is a continuous function, and $\emph{mdim}(\hyp{K}(X), \induce{T}) = 0$, then $\emph{mdim}(X, T) = 0$.
\end{proposition}
\begin{proof}
    Consider $\Psi: X \to \hyp{K}(X)$ given by $\Psi(x) = \{x\}$. Then $\Psi$ is injective and continuous. Furthermore, $\Psi$ is a topological conjugation, that is, 
    \begin{equation*}
        \Psi\circ T(x) = \{Tx\} = \induce{T}(\{x\}) = \induce{T}\circ\Psi(x).
    \end{equation*}
    Therefore, $0 = \text{mdim}(\hyp{K}(X), \induce{T}) \geq \text{mdim}(X, T)$.
\end{proof}

In other words, since $(X, T)$ is always a factor of $(\hyp{K}(X), \induce{T})$, the above Proposition is a direct consequence of the following fact:\\

\begin{lemma}
    Given $(X, d)$ a compact metric space and $T: X \to X$ a continuous map, $\emph{mdim}(\hyp{K}(X), \induce{T}) \geq \emph{mdim}(X, T)$.
\end{lemma}


\section{Metric mean dimension of induced maps}\label{sec:metric-mean-dim-induced}

For an arbitrary map, $T: X \to X$, the metric mean dimension offers a viable alternative when it is not yet possible to calculate the mean dimension of its induced hyperspace map. The calculation of the metric mean dimension is simpler because it relies on how the dynamics separates points in the phase space.

To compute this dimension in the hyperspace $\hyp{K}(X)$, we must first define a dynamical distance based on the Hausdorff metric $d_H$. Precisely, for $n \in \mathbb{N}$ and $A, B \in \hyp{K}(X)$, the \emph{dynamical Hausdorff metric} for $\induce{T}: \hyp{K}(X) \to \hyp{K}(X)$ is defined as
\begin{equation*}
    \mathcal{H}_n(A, B) = \max\{ d_H(\induce{T}^iA, \induce{T}^iB); 0 \leq i \leq n - 1\},
\end{equation*}
and for $\varepsilon > 0$, the \emph{Hausdorff $(n, \varepsilon)$-dynamical ball} around $A \in \hyp{K}(X)$ is defined as
\begin{equation*}
    \mathbf{B}^H_{(n, \varepsilon)}(A) = \{B \in \hyp{K}(X); \mathcal{H}_n(A, B) < \varepsilon\}.
\end{equation*}

Before showing some properties of the dynamical Hausdorff metric, we need to establish certain topological properties of locally connected continuum spaces. We begin with the following well-known result.\\

\begin{proposition}[\cite{illanes_nadler_1999}, Proposition 10.4]\label{prop:loc-conn-isometric-arc}
    Let $X$ be a locally connected continuum with a convex metric $d$. Then any two points, $x, y \in X$, can be joined by an arc $\gamma$ in $X$ such that $\gamma$ is isometric to the closed interval $[0, d(x, y)]$.\\
\end{proposition}

As a consequence of the above proposition, we obtain the expected result regarding the strict inclusion of open sets when there is a strict inequality between their diameters.\\

\begin{lemma}
    Let $(X, d)$ be a locally connected continuum where $d$ is a convex metric. For sufficiently small $\varepsilon > 0$, for all $0 < \varepsilon_0 < \varepsilon$ and for any $x \in X$, $B_{\varepsilon_0}(x) \subsetneq B_\varepsilon(x)$.
\end{lemma}
\begin{proof}
    Let $\varepsilon > 0$ be such that $2\varepsilon < \text{diam}(X)$. Hence, for any $x \in X$, there is $y \in X$ such that $y \notin B_\varepsilon(x)$. Since $X$ is a locally connected continuum, there is an arc $\gamma$ from $x$ to $y$ such that $\gamma$ is isometric to the interval $[0, d(x, y)]$. Hence, given $\varepsilon_0 < \varepsilon$, there is $z \in \gamma$ such that $\varepsilon_0 < d(x, z) < \varepsilon$. Therefore, $z \in B_\varepsilon(x)$ and $z \notin B_{\varepsilon_0}(x)$.
\end{proof}

\begin{remark}\label{remark:isometric-curve-contained-open-ball}
    With the same hypotheses as the above lemma, if $y \in B_\varepsilon(x)$ and $\gamma$ is an arc connecting $x$ to $y$ such that it is isometric to the interval $[0, d(x, y)]$, then $d(x, z) < \varepsilon$ for all $z \in \gamma$. Indeed, let $\psi: \gamma \to [0, d(x, y)]$ be an isometry. Since $y \in B_\varepsilon(x)$, we know that $d(x, y) < \varepsilon$. For any $z \in \gamma$, the distance $d(x, z)$ corresponds to the length of the image of the sub-arc $\gamma_{[x, z]}$ between $x$ and $z$ under $\psi$. This length cannot exceed the length of the entire image interval, so $d(x,z) \leq d(x,y) < \varepsilon$. Therefore, $\gamma \subset B_\varepsilon(x)$.\\
\end{remark}

Given $x \in X$, recall that $B_{(n, \varepsilon)}(x)$ is the $(n, \varepsilon)$-dynamicall ball. For $A \in \hyp{K}(X)$, define the open set $\mathfrak{B}_{(n, \varepsilon)}(A)$ of $\hyp{K}(X)$ as
\begin{equation*}
    \mathfrak{B}_{(n, \varepsilon)}(A) = \{C \in \hyp{K}(X); C \subset \bigcup\limits_{x \in A} B_{(n, \varepsilon)}(x) \text{ and } C \cap B_{(n, \varepsilon)}(x) \neq \varnothing \text{ for all } x \in A\}.
\end{equation*}
It is an open set because $B_{(n, \varepsilon)}(x)$ is an open set of $X$, hence $\mathfrak{B}_{(n, \varepsilon)}(A)$ is an open set in the Vietoris topology of the hyperspace, defined in \S\ref{section:hyperspaces-induced-maps}.\\

\begin{lemma}\label{lemma:dynamical-balls-hyperspace-subset}
    Let $(X, d)$ be a locally connected continuum where $d$ is a convex metric. For $A \in \hyp{K}(X)$, $n \in \mathbb{N}$ and sufficiently small $\varepsilon > 0$, $\mathfrak{B}_{(n, \varepsilon)}(A) \subsetneq \emph{\textbf{B}}^H_{(n, \varepsilon)}(A)$.
\end{lemma}
\begin{proof}
    As a consequence of the first hypothesis, for all $i \in \{0, ..., n - 1\}$,
    \begin{equation*}
        T^i(C) \subset \bigcup\limits_{x \in A} B_\varepsilon(T^ix).
    \end{equation*}
    Fix an index $i$ and assume, in contradiction, that for every $\varepsilon_0 < \varepsilon$ the above inclusion is not true. That is, for every $\varepsilon_0 < \varepsilon$, there exists $y_0 \in T^i(C)$ such that $y_0 \notin \bigcup_{x \in A} B_{\varepsilon_0}(T^ix)$. Let $(\varepsilon_n)_{n = 1}^\infty$ be an increasing sequence such that $\varepsilon_n \to \varepsilon$. Then, there exists a sequence $(y_n)_{n = 1}^\infty$ such that each $y_n \in T^i(C)$ and $y_n \notin \bigcup_{x \in A} B_{\varepsilon_n}(T^ix)$. Since $T^i(C)$ is a closed set, there exists a subsequence $y_{n_k}$ such that $y_{n_k} \to y$ and $y \in T^i(C)$. Hence, there is $z \in A$ such that $y \in B_\varepsilon(T^iz)$. Let $\eta > 0$ be such that $\eta < \varepsilon$ and $y \in B_\eta(T^iz)$. Since $(X, d)$ is a locally connected continuum with a convex metric $d$, $B_\eta(T^iz) \subsetneq B_\varepsilon(T^iz)$. Let $N \in \mathbb{N}$ be such that for all $k \geq N$, $\varepsilon_{n_k} > \eta$. This implies that $y_{n_k} \notin B_\eta(T^iz)$ for all $k \geq N$. Contradiction with the fact that $y_{n_k} \to y$. Thus, for each $i \in \{0, ..., n-1\}$, there exists $\delta_i \in (0, \varepsilon)$ such that
    \begin{equation*}
        T^i(C) \subset \bigcup\limits_{x \in A} B_{\delta_i}(T^ix).
    \end{equation*}
    Set $\delta = \max_{0 \leq i \leq n - 1} \delta_i$. Note that $\delta \in (0, \varepsilon)$. Hence, $T^i(C) \subset \bigcup_{x \in A} B_\delta(T^ix)$ for all $i \in \{0, ..., n - 1\}$. Moreover, fix and index $i$ and suppose, in contradiction, that for every $\varepsilon_0 < \varepsilon$ there exists $x_0 \in A$ such that $T^i(C) \cap B_{\varepsilon_0}(T^ix_0) = \varnothing$. If $(\varepsilon_n)_{n = 1}^\infty$ is an increasing sequence such that $\varepsilon_n \to \varepsilon$, then there exists a sequence $(x_n)_{n = 1}^\infty$ such that $x_n \in A$ and $T^i(C) \cap B_{\varepsilon_n}(T^ix_n) = \varnothing$. Since $A$ is a closed set, there is a subsequence $x_{n_k}$ such that $x_{n_k} \to x$ and $x \in A$. By the hypothesis, $T^i(C) \cap B_\varepsilon(T^ix) \neq \varnothing$. For a given $z \in T^i(C) \cap B_\varepsilon(T^ix)$, there exists $\eta \in (0, \varepsilon)$ such that $z \in B_\eta(T^ix)$. Let $N \in \mathbb{N}$ be such that for all $k \geq N$, $d(x_{n_k}, x) < \varepsilon_{n_k} - \eta$. This is possible because $\varepsilon_{n_k} \to \varepsilon$ and $x_{n_k} \to x$. Hence, 
    \begin{equation*}
        d(z, x_{n_k}) \leq d(z, x) + d(x_{n_k}, x) < \eta + d(x_{n_k}, x) < \varepsilon_{n_k}.
    \end{equation*}
    Contradiction with the fact that $T^i(C) \cap B_{\varepsilon_{n_k}}(T^ix_{n_k}) = \varnothing$. Therefore, for all $i \in \{0, ..., n - 1\}$, there exists $\eta_i \in (0, \varepsilon)$ such that for all $x \in A$, $T^i(C) \cap B_{\eta_i}(x) \neq \varnothing$. Let $\theta_i = \max\{\delta_i, \eta_i\}$, then
    \begin{equation*}
        T^i(C) \subset \bigcup\limits_{x \in A} B_{\theta_i}(T^ix) \text { and } T^i(C) \cap B_{\theta_i}(x) \neq \varnothing,
    \end{equation*}
    for all $x \in A$. Furthermore, for $\theta = \max_{0 \leq i \leq n - 1} \theta_i$, then, for all $i \in \{0, ..., n - 1\}$ and for any $x \in A$, 
    \begin{equation*}
        T^i(C) \subset \bigcup\limits_{x \in A} B_{\theta}(T^ix) \text { and } T^i(C) \cap B_{\theta}(x) \neq \varnothing.
    \end{equation*}
    In other words, $\induce{T}^iC \subset (\induce{T}^iA)_\theta$ and $\induce{T}^iA \subset (\induce{T}^iC)_\theta$. Since $\theta \in (0, \varepsilon)$, for all $i \in \{0, ..., n - 1\}$, $d_H(\induce{T}^iC, \induce{T}^iA) < \varepsilon$. Finally, $\mathcal{H}_n(A, C) < \varepsilon$.
\end{proof}

We will provide an example to show that the set inclusion in the statement of the above lemma is not an equality.\\

\begin{example}
    Consider the circle $S^1$ and choose two distinct points $p$ and $q$ in $S^1$. Denote the two distinct arcs that connect $p$ to $q$ as $\gamma_p$ and $\gamma_q$. Construct a homeomorphism $H$ on $S^1$ such that $p$ and $q$ are fixed points, $\gamma_p$ is the stable manifold of $p$, and $\gamma_q$ is the stable manifold of $q$. Given a sufficiently small $\varepsilon > 0$, consider $x \in \gamma_q$ such that $d(x, p) < \varepsilon$, and consider $y \in \gamma_p$ such that $d(q, y) < \varepsilon$, where $d$ is a convex metric for the circle. Hence $d_H(\{x, y\}, \{p, q\}) < \varepsilon$. Since $x \in \gamma_q$, we can assume that $d(Hx, q) < \varepsilon$ and since $y \in \gamma_p$, it is possible to consider that $d(Hy, p) < \varepsilon$. For sufficiently small $\varepsilon > 0$, this implies that $d(Hx, p) \geq \varepsilon$. Thus, $x \notin B_{(2, \varepsilon)}(p)$, the $(2, \varepsilon)$-dynamical ball around $p$. With a similar argument, $y \notin B_{(2, \varepsilon)}(q)$. Therefore, $\{x, y\} \notin \mathfrak{B}_{(2, \varepsilon)}(\{p, q\})$. On the other hand, $d_H(\{Hx, Hy\}, \{p, q\}) < \varepsilon$. That is, $\{x, y\} \in \mathbf{B}^H_{(2, \varepsilon)}(\{p, q\})$.\\
\end{example}

In simpler terms, for $A, C \in \hyp{K}(X)$, the Lemma \ref{lemma:dynamical-balls-hyperspace-subset} shows that if we consider the metric $\mathcal{H}^n$ on $\hyp{K}(X)$, defined by 
\begin{equation*}
    \mathcal{H}^n(A, C) = \inf\{\varepsilon > 0; C \subset \bigcup\limits_{x \in A} B_{(n, \varepsilon)}(x) \text{ and } A \subset \bigcup\limits_{y \in C} B_{(n, \varepsilon)}(y) \},
\end{equation*}
then $\mathcal{H}_n(A, C) \leq \mathcal{H}^n(A, C)$.\\

\begin{lemma}\label{lemma:span-hyperspace-bounded-exponential-span-base}
    Let $(X, d)$ be a locally connected continuum where $d$ is a convex metric, and consider $T: X \to X$ a map. For sufficiently small $\varepsilon > 0$ and $n \in \mathbb{N}$,
    \begin{equation*}
        \emph{Span}(\induce{T}, n, \varepsilon) \leq 2^{\emph{Span}(T, n, \varepsilon)} - 1.
    \end{equation*}
\end{lemma}
\begin{proof}
    Let $E = \{p_1, ..., p_m\}$ be a $(n, \varepsilon)$-spanning set of minimal cardinality for $T$. That is, $m = \text{Span}(T, n, \varepsilon)$. For any closed nonempty set $A \subset X$, there exists $E_0 \subset E$ such that 
    \begin{equation*}
        A \subset \bigcup\limits_{x \in E_0} B_{(n, \varepsilon)}(x) \text{ and } A \cap B_{(n, \varepsilon)}(x) \neq \varnothing,
    \end{equation*}
    for all $x \in E_0$. Hence, by Lemma \ref{lemma:dynamical-balls-hyperspace-subset}, $\mathcal{H}_n(A, E_0) < \varepsilon$. Therefore, there exists at most $2^m - 1$ elements for the $(n, \varepsilon)$-spanning set of $\induce{T}$. 
\end{proof}

The next result is a direct consequence of the above lemma. It states that if the complexity of $T$, measured in terms of polynomial entropy, is sufficiently low, then the topological entropy of $T$ is zero.\\

\begin{lemma}\label{lemma:positive-entropy-induced-hyperspace-poly-entropy-one}
    Consider $X$ a locally connected continuum endowed with a convex metric $d$ and $T: X \to X$ a map. If $h(\induce{T}) > 0$, then $h_{\emph{pol}}(T) \geq 1$.
\end{lemma}
\begin{proof}
    Let $\alpha > 0$ be such that $h(\induce{T}) > \alpha$ and $\varepsilon_0 > 0$ be such that for all $\varepsilon \in (0, \varepsilon_0)$,
    \begin{equation*}
        \widetilde{h}(\induce{T}, \varepsilon) > \alpha. 
    \end{equation*}
    In other words, 
    \begin{equation*}
        \limsup\limits_{n \to \infty} \frac{\log \text{Span}(\induce{T}, n, \varepsilon)}{n} > \alpha.
    \end{equation*}
    This implies that there exists sequence $(n_j)_{j \in \mathbb{N}}$ such that for all $j \in \mathbb{N}$
    \begin{equation*}
        \frac{\log \text{Span}(\induce{T}, n_j, \varepsilon)}{n_j} > \alpha.
    \end{equation*}
    By Lemma \ref{lemma:span-hyperspace-bounded-exponential-span-base}, 
    \begin{equation*}
        \text{Span}(T, n_j, \varepsilon)\cdot\log 2 \geq \log \text{Span}(\induce{T}, n_j, \varepsilon) > n_j\cdot\alpha.
    \end{equation*}
    Note that $n_j \cdot \alpha > 0$, because both $n_j$ and $\alpha$ are positive real numbers. Thus, taking the $\log$ on both sides of the above inequality,
    \begin{equation*}
        \log \text{Span}(T, n_j, \varepsilon) + \log\log 2 > \log n_j + \log\alpha.
    \end{equation*}
    We can assume that $n_j \geq 2$, for all $j \in \mathbb{N}$. Dividing by $\log n_j$ on both sides,
    \begin{equation*}
        \frac{\log \text{Span}(T, n_j, \varepsilon)}{\log n_j} + \frac{\log\log 2}{\log n_j} > 1 + \frac{\log\alpha}{\log n_j}.
    \end{equation*}
    This implies
    \begin{equation*}
        \limsup\limits_{n \to \infty} \frac{\log \text{Span}(T, n, \varepsilon)}{\log n} \geq 1.
    \end{equation*}
    Therefore,
    \begin{equation*}
        h_{\text{pol}}(T) = \lim\limits_{\varepsilon \to 0}\limsup\limits_{n \to \infty} \frac{\log \text{Span}(T, n, \varepsilon)}{\log n} \geq 1.
    \end{equation*}
\end{proof}

Thus, by Lemma \ref{lemma:positive-entropy-induced-hyperspace-poly-entropy-one}, if $h_{\text{pol}}(T) < 1$, then $h(\induce{T}) = 0$. Therefore, by Equation \ref{eq:mean-metric-mean-inequality},
\begin{equation*}
    0 = h(\induce{T}) \geq \text{mdim}_M(\hyp{K}(X), d_H, \induce{T}) \geq \text{mdim}(\hyp{K}(X), \induce{T}).
\end{equation*}
The above inequality shows Theorem \ref{thm:poly-entropy-less1-metric-mean-dim-zero}. Hence, if the complexity of the base map $T$ is sufficiently low, then the phenomenon of mean dimension explosion does not occur for the induced map $\induce{T}$.

Another way to guarantee that the metric mean dimension of the induced hyperspace map is zero is for the base map to be an isometry. In fact, by Remark \ref{prop:isometry-hyp}, if $T$ is an isometry, then $\induce{T}$ is also an isometry. This implies that $h(\induce{T}) = 0$, and again by the above inequality, $\text{mdim}_M(\hyp{K}(X), d_H, \induce{T}) = 0$.

For the following, consider $\varepsilon > 0$ and let $N_\varepsilon(X)$ be the smallest number of open sets of diameter at most $\varepsilon$ which can cover $X$. Observe that $N_\varepsilon(X)$ depends on the metric $d$ used to compute the diameter of the open sets.\\

\begin{lemma}\label{lemma:number-cover-epsilon-hyperspace-exponential-upper-bound}
    If $(X, d)$ is a compact metric space and $\varepsilon > 0$, then
    \begin{equation}\label{eq:upper-bound-number-cover-hyperspace}
        N_\varepsilon(\hyp{K}(X)) \leq 2^{N_\varepsilon(X)} - 1.
    \end{equation}
\end{lemma}
\begin{proof}
    Consider a cover $U_1, \ldots, U_m$ of $X$ such that $m = N_\varepsilon(X)$ and $\text{diam}(U_i) < \varepsilon$ for all $1 \leq i \leq m$. Since the topology of $\hyp{K}(X)$ is given by the Vietoris topology, 
    \begin{equation*}
        V_{i_1, \ldots, i_k} = \{A \in \hyp{K}(X); A \subset \bigcup_{j = 1}^k U_{i_j} \text{ and } A \cap U_{i_j} \neq \varnothing \text{ for all } 1 \leq j \leq k\}
    \end{equation*}
    forms an open cover of $\hyp{K}(X)$. Indeed, for any $A \in \hyp{K}(X)$ there exists a subcover $U_{i_1}, \ldots, U_{i_\ell}$ such that $A \subset \bigcup_{j = 1}^\ell U_{i_j}$. We may also assume that for any strict subset $P \subsetneq \{i_1, \ldots, i_\ell\}$, $A \nsubseteq \bigcup_{p \in P} U_{p}$. This implies $A \cap U_{i_j} \neq \varnothing$ for all $1 \leq j \leq \ell$. Hence $A \in V_{i_1, \ldots, i_\ell}$. Moreover, choose $A, B \in V_{i_1, \ldots, i_k}$. For each $1 \leq j \leq k$, take $x \in A \cap U_{i_j}$ and $y \in B \cap U_{i_j}$. Then $d(x, y) \leq \varepsilon$. Since this holds for all $j \in \{1, \ldots, k\}$, it follows that $A \subset B_\varepsilon$ and $B \subset A_\varepsilon$. Thus $d_H(A, B) \leq \varepsilon$, so $\text{diam}(V_{i_1, \ldots, i_k}) \leq \varepsilon$. Observe that there are $2^m - 1$ open sets of the form $V_{i_1, \ldots, i_k}$. Therefore, $N_\varepsilon(\hyp{K}(X)) \leq 2^m - 1$. 
\end{proof}

\begin{remark}\label{rmk:separated-sets-hyperspace}
In the same spirit as the above lemma, if $E \subset X$ is a finite set such that for any distinct $x, y \in E$, we have $d(x, y) \geq \varepsilon$, then for any distinct nonempty subsets $A, B \subset E$, we have $d_H(A, B) \geq \varepsilon$. Indeed, since $A \neq B$, assume without loss of generality that there is an element $x \in A \setminus B$. By hypothesis, $d(x,y) \geq \varepsilon$ for all $y \in B$, which implies $d_H(A,B) \geq \varepsilon$. To conclude, note that there exist at least $2^{\#E} - 1$ sets which are $\varepsilon$-separated by $d_H$.\\
\end{remark}

Recall that in \S\ref{subsec:difference-mean-top-dim} we discussed that the topological dimension of $\hyp{K}(X)$ is zero whenever $X$ is a zero-dimensional $T_1$-space. This does not necessarily imply that other concepts of dimension for a topological space must also be zero. Nevertheless, we construct a zero-dimensional compact metric space such that the box-counting dimension of its hyperspace is also zero. For a compact metric space $(X, d)$, the \emph{box-counting dimension} of $X$ is defined as
\begin{equation*}
    \dim_B(X, d) = \lim_{\varepsilon \to 0} \frac{\log N_\varepsilon(X)}{-\log \varepsilon},
\end{equation*}
whenever this limit exists.\\

\begin{lemma}\label{lemma:cantor-set-hyperspace-zero-box-dimension}
    Given the usual Euclidean metric $\rho$ on $[0, 1]$, there exists a Cantor set $\Lambda \subset [0, 1]$ such that $\dim_B(\hyp{K}(\Lambda), \rho_H) = 0$.
\end{lemma}
\begin{proof}
    For $k \in \mathbb{N}$, $k \geq 2$, consider the following construction of a Cantor set: in the first step, remove an open interval $J_{1, 1} \subset [0, 1]$ such that $J_{1, 1}^\complement \cap [0, 1] = I_{1, 0} \cup I_{1, 1}$, where $\text{diam}(I_{1, 0}) = \text{diam}(I_{1, 1}) = 2^{-4}$, and call $I_{1, 0} \cup I_{1, 1} = \Lambda_1$. In step $k$, remove the open intervals $J_{k, 1}, ..., J_{k, 2^{k - 1}}$ from $\Lambda_{k - 1}$ such that
    \begin{equation*}
        J_{k, 1}^\complement \cap ... \cap J_{k, 2^{k - 1}}^\complement \cap \Lambda_{k - 1} = I_{k, 1} \cup ... \cup I_{k, 2^k} = \Lambda_k,
    \end{equation*}
    where $\text{diam}(I_{k, j}) = 2^{-(4^k)}$, for $1 \leq j \leq 2^k$. The Cantor set $\Lambda$ is given by $\Lambda = \lim_{k \to \infty} \Lambda_k$. To show that the box-counting dimension of $\hyp{K}(\Lambda)$ is equal to $0$, observe that the minimum number of open sets of diameter at most $2^{-(4^k)}$ required to cover $\Lambda$ is $2^{k+1}$. By Equation (\ref{eq:upper-bound-number-cover-hyperspace}) in Lemma \ref{lemma:number-cover-epsilon-hyperspace-exponential-upper-bound}, 
    \begin{equation*}
        N_k(\hyp{K}(\Lambda)) \leq 2^{N_k(\Lambda)} - 1 \leq 2^{2^{k + 1}},
    \end{equation*}
    where $N_k$ is the smallest number of open sets of diameter at most $2^{-(4^k)}$. Therefore,
    \begin{equation*}
        \dim_B(\hyp{K}(\Lambda), \rho_H) \leq \lim_{k \to \infty} \frac{\log 2^{2^{k + 1}}}{-\log 2^{-(4^k)}} = \lim_{k \to \infty} \frac{2^{k + 1}}{4^k} = 0. 
    \end{equation*}
\end{proof}

The box-counting dimension is strongly connected to the metric mean dimension for any map $T$: it provides an upper bound. More precisely, we have the inequality
\begin{equation}\label{eq:box-counting-bounds-metric-mean-dim}
    \text{mdim}_M(X, d, T) \leq \dim_B(X, d)
\end{equation}
whenever both dimensions exist, which is a specific case of a more general result. Proved in Remark 4 of \cite{velozo_velozo_2017}, this inequality, along with the above lemma, allows us to construct an example of a map whose topological entropy is positive, but whose induced hyperspace map has a metric mean dimension of zero.

\begin{proof}[Proof of Theorem \emph{\ref{thm:existence-compact-space-pos-entropy-zero-metric-mean-dim}}]
    Consider the Cantor set $\Lambda \subset [0, 1]$ from Lemma \ref{lemma:cantor-set-hyperspace-zero-box-dimension}. It is a compact metric space endowed with the Euclidean metric $\rho$. Given the one-sided symbolic space $\Sigma_2$, there exists a homeomorphism between $\Lambda$ and $\Sigma_2$. Hence, there exists a continuous map $T: \Lambda \to \Lambda$ that is conjugate to the one-sided shift $\sigma: \Sigma_2 \to \Sigma_2$. Indeed, if $\Psi: \Lambda \to \Sigma_2$ is a homeomorphism, then define $T$ as $T = \Psi^{-1}\circ\sigma\circ\Psi$. This implies $h(T) > 0$. Finally, by Equation (\ref{eq:box-counting-bounds-metric-mean-dim}), the metric mean dimension of $\induce{T}$ is bounded above by the box-counting dimension of $\hyp{K}(\Lambda)$, and therefore $\text{mdim}_M(\hyp{K}(X), \rho_H, \induce{T}) = 0$.  
\end{proof}

Given a surjective map $T: X \to X$ and a finite set of points $E \subset X$. A \emph{first trail of $E$} is a set $E_1 \subset T^{-1}(E)$ such that $\#E_1 = \#E$ and $T(E_1) = E$. Moreover, for $n \geq 2$, an \emph{$n$-th trail of $E$} is a set $E_n$ such that $\#E_n = \#E_{n - 1}$ and $T(E_n) = E_{n - 1}$. In particular, $T^n(E_n) = E$.

In what follows, we prove a result similar to Theorem \ref{thm:nonwandering-infty-mean-dim}. Precisely, we show that if $X$ is a locally connected continuum endowed with a convex metric $d$ and $T:X \to X$ is a surjective map such that $\Omega(T) \subsetneq X$, then $\text{mdim}_M(\hyp{K}(X), d_H, \induce{T}) = \infty$.

\begin{proof}[Proof of Theorem \emph{\ref{thm:surjective-map-nonwandering-infinite-metric-mean-dim}}]
    Choosing $x \in X\cap\Omega(T)^\complement$, then $x$ is a wandering point of $X$. Hence, there exists a neighborhood $U$ of $x$ such that for all $n \geq 1$, $T^n(U) \cap U = \varnothing$ and $T^{-n}(U) \cap U = \varnothing$. Since $U^\complement$ is compact, $\eta := \inf_{x_0 U^\complement} d(x, x_0)$ is a positive real number. This implies $B_\eta(x) \subset U$. Given $y \in B_\eta(x)$, there exists an arc $\gamma$ connecting $x$ to $y$ that is isometric to $[0, d(x, y)]$, by Proposition \ref{prop:loc-conn-isometric-arc}. Furthermore, by Remark \ref{remark:isometric-curve-contained-open-ball}, $\gamma \subset B_\eta(x)$. Clearly, $\gamma \subset U$. 
    
    Since $U$ is an open set and $\gamma$ is compact, there exists $\varepsilon_0 > 0$ such that for all $\varepsilon \in (0, \varepsilon_0)$ and for all $z \in \gamma$, $B_\varepsilon(z) \subset U$. Set $\alpha_0 := d(x, y)$ and choose $\lfloor \frac{\alpha_0}{\varepsilon} \rfloor$ distinct points in $\gamma$ such that the distance between any pair of points is greater than or equal to a sufficiently small $\varepsilon$. That is, since $\gamma$ is isometric to $[0, \alpha_0]$, there exists a set $E = \{x_1, ..., x_M\} \subset \gamma$, where $M = \lfloor \frac{\alpha_0}{\varepsilon} \rfloor$, such that $d(x_i, x_j) \geq \varepsilon$, for $1 \leq i < j \leq M$. Observe that, by Remark \ref{rmk:separated-sets-hyperspace}, for distinct nonempty sets $A, B \subset E$, $d_H(A, B) \geq \varepsilon$. 

    Denote all the nonempty subsets of $E$ by $A^i_0$, for $1 \leq i \leq 2^M - 1$, and the $n$-th trail of $A^i_0$ as $A^i_n$. Given $(i_0, i_1, ..., i_{n - 1}) \in \{1, ..., 2^M - 1\}^n$, define
    \begin{equation}\label{eq:def-trail-list-set}
        A_{i_0, ..., i_{n - 1}} = \bigcup\limits_{k = 0}^{n - 1} A^{i_k}_k.
    \end{equation}
    Observe that $A_{i_0, ..., i_{n - 1}}$ is a point of $\hyp{K}(X)$ because it is a finite union of compact sets. Moreover, $A^{i_k}_k \subset T^{-k}(U)$, hence, for $k \geq 1$, the Hausdorff distance from $A^{i_k}_k$ to any subset of $E$ is always greater than or equal to $\varepsilon$. For distinct $(j_0, j_1, ..., j_{n - 1}) \in \{1, ..., 2^M - 1\}^n$, we claim that
    \begin{equation}\label{eq:trail-sets-epsilon-separated}
        \mathcal{H}_n(A_{i_0, ..., i_{n - 1}}, A_{j_0, ..., j_{n - 1}}) \geq \varepsilon,
    \end{equation}
    which we will assume is true for now. This implies $\text{Sep}(\induce{T}, n, \varepsilon) \geq (2^M - 1)^n$. Thus,
    \begin{equation*}
        h(\induce{T}, \varepsilon) = \limsup\limits_{n \to \infty} \frac{\log \text{Sep}(\induce{T}, n, \varepsilon)}{n} \geq M\log 2.
    \end{equation*}
    Therefore,
    \begin{equation*}
        \underline{\text{mdim}_M}(\hyp{K}(X), d_H, \induce{T}) = \liminf\limits_{\varepsilon \to 0} \frac{h(\induce{T}, \varepsilon)}{-\log\varepsilon} \geq \liminf\limits_{\varepsilon \to 0}\frac{M\log 2}{-\log \varepsilon} = \liminf\limits_{\varepsilon \to 0} \left\lfloor\frac{\alpha_0}{\varepsilon}\right\rfloor\frac{\log 2}{-\log \varepsilon} = \infty.
    \end{equation*}
    To conclude this proof, we only need to show that Equation (\ref{eq:trail-sets-epsilon-separated}) is valid. Consider two distinct elements $(i_0, i_1, ..., i_{n - 1})$ and $(j_0, j_1, ..., j_{n - 1})$ of $\{1, ..., 2^M - 1\}^n$. Hence, there is $k \in \{0, ..., n - 1\}$ such that $i_k \neq j_k$. That is, by Equation (\ref{eq:def-trail-list-set}), the $k$-th trail of $A^{i_k}$ belongs to $A_{i_0, ..., i_{n - 1}}$, and the $k$-th trail of $A^{j_k}$ belongs to $A_{j_0, ..., j_{n - 1}}$. In other words,
    \begin{equation*}
        A^{i_k} \subset T^k(A_{i_0, ..., i_{n - 1}}) \text{ and } A^{j_k} \subset T^k(A_{j_0, ..., j_{n - 1}}).
    \end{equation*}
    Since $A^{i_k}$ and $A^{j_k}$ are distinct subsets of $E$, then $d_H(A^{i_k}, A^{j_k}) \geq \varepsilon$, by Remark \ref{rmk:separated-sets-hyperspace}. Therefore,
    \begin{equation*}
        d_H(T^k(A_{i_0, ..., i_{n - 1}}), T^k(A_{j_0, ..., j_{n - 1}})) \geq \varepsilon.
    \end{equation*}
    In fact, consider $k_0$ distinct of $k$ in $\{0, ..., n - 1\}$, then both $T^k(A_{k_0}^{i_{k_0}})$ and $T^k(A_{k_0}^{j_{k_0}})$ belongs to $T^{k - k_0}(U)$, where $T^{k - k_0}(U) \cap U = \varnothing$, because $k - k_0 \neq 0$. This implies
    \begin{equation*}
        \min\{d_H(T^k(A_{k_0}^{i_{k_0}}), A^{i_k}), d_H(T^k(A_{k_0}^{i_{k_0}}), A^{j_k}), d_H(T^k(A_{k_0}^{j_{k_0}}), A^{i_k}), d_H(T^k(A_{k_0}^{j_{k_0}}), A^{j_k})\} \geq \varepsilon.
    \end{equation*}
    Thus,
    \begin{equation*}
        d_H(T^k(A_{i_0, ..., i_{n - 1}}), T^k(A_{j_0, ..., j_{n - 1}})) \geq d_H(A^{i_k}, A^{j_k}) \geq \varepsilon.
    \end{equation*}
\end{proof}

The hypothesis that $X$ is a locally connected continuum is essential, as it is a necessary condition for the metric mean dimension of $\induce{T}$ to attain infinity with respect to the Hausdorff metric induced by some metric $d$.

\section{Final Remarks}

The mean dimension explosion phenomenon indicates that the induced hyperspace map has too many different orbits; hence, its complexity is overwhelmed by even slight separation of points in the base system. Future research may be directed toward classifying zero topological entropy systems with finite (metric) mean dimension of their induced maps.

\subsection{Conflict of Interests} On behalf of all authors, the corresponding author states that there is no conflict of interest.


\bibliography{sn-bibliography}

\end{document}